\documentclass[12pt]{amsart}

\usepackage[english]{babel}
\usepackage[utf8x]{inputenc}
\usepackage[T1]{fontenc}

\usepackage[margin=1.15in]{geometry}

\usepackage{amsmath}
\usepackage{amstext}
\usepackage{amsfonts}
\usepackage{amssymb, dsfont}
\usepackage{amsthm}
\usepackage{amsrefs}
\usepackage{mathtools}
\usepackage{color}
\usepackage{graphicx}
\usepackage[colorlinks=true, allcolors=blue]{hyperref}
\usepackage{float}
\allowdisplaybreaks
\usepackage{tikz-cd}
\usepackage{amsbsy}
\usepackage{amscd}
\usetikzlibrary{matrix,arrows,decorations.pathmorphing}
\usepackage{enumitem}
\usepackage{setspace}

\usepackage{amsrefs}

\newtheorem{thm}{Theorem}[section]
\newtheorem{lem}[thm]{Lemma}
\newtheorem{cor}[thm]{Corollary}
\newtheorem{prop}[thm]{Proposition}

\theoremstyle{definition}
\newtheorem{example}[thm]{Example}
\theoremstyle{definition}
\newtheorem{examples}[thm]{Examples}
\theoremstyle{definition}
\newtheorem{defn}[thm]{Definition}
\theoremstyle{definition}
\newtheorem{remark}[thm]{Remark}

\newcommand{\mc}[1]{\mathcal{#1}}
\newcommand{\e}[1]{\emph{#1}}

\newcommand{\la}{\langle}
\newcommand{\ra}{\rangle}

\newcommand{\rmv}[1]{}

\newcommand{\LO}{L^1(G)}

\newcommand{\LT}{L^2(G)}

\newcommand{\LI}{L^{\infty}(G)}

\newcommand{\BH}{\mc{B}(H)}
\newcommand{\Th}{\mc{T}(H)}

\newcommand{\BLT}{\mc{B}(L^2(G))}

\newcommand{\TC}{\mc{T}(L^2(G))}

\newcommand{\vphi}{\varphi}

\newcommand{\lm}{\lambda}

\newcommand{\om}{\omega}

\newcommand{\ten}{\otimes}
\newcommand{\oten}{\overline{\otimes}}
\newcommand{\pten}{\widehat{\otimes}}

\newcommand{\iten}{\otimes^{\vee}}

\newcommand{\opten}{\widehat{\otimes}}

\newcommand{\id}{\textnormal{id}}
\newcommand{\h}[1]{\widehat{#1}}
\newcommand{\wh}[1]{\widehat{#1}}

\newcommand{\Ad}{\mathrm{Ad}}

\providecommand{\norm}[1]{\lVert#1\rVert}

\newcommand{\Amod}{\mathbf{{A}mod}}
\newcommand{\modA}{\mathbf{mod{A}}}

\newcommand{\bC}{{\mathbb{C}}}

\newcommand{\bR}{{\mathbb{R}}}

\newcommand{\cB}{{\mathcal{B}}}

\newcommand{\G}{\Gamma}

\newcommand{\cP}{\mathcal{P}}

\newcommand{\cI}{{\mathcal{I}}}

\newcommand{\acts}{\curvearrowright}
\newcommand{\act}{\curvearrowright}

\begin{document}

\title[]{A weak expectation property for operator modules, injectivity and amenable actions}
\author{Alex Bearden}
\email{cbearden@uttyler.edu}
\address{Department of Mathematics, University of Texas at Tyler, Tyler, TX 75799}
\author{Jason Crann}
\email{jasoncrann@cunet.carleton.ca}
\address{School of Mathematics and Statistics, Carleton University, Ottawa, ON K1S 5B6}


\begin{abstract} We introduce an equivariant version of the weak expectation property (WEP) at the level of operator modules over completely contractive Banach algebras $A$. We prove a number of general results---for example, a characterization of the $A$-WEP in terms of an appropriate $A$-injective envelope, and also a characterization of those $A$ for which $A$-WEP implies WEP. In the case of $A=\LO$, we recover the $G$-WEP for $G$-$C^*$-algebras in recent work of Buss--Echterhoff--Willett \cite{BEW}. When $A=A(G)$, we obtain a dual notion for operator modules over the Fourier algebra. These dual notions are related in the setting of dynamical systems, where we show that a $W^*$-dynamical system $(M,G,\alpha)$ with $M$ injective is amenable if and only if $M$ is $\LO$-injective if and only if the crossed product $G\bar{\ltimes}M$ is $A(G)$-injective. Analogously, we show that a $C^*$-dynamical system $(A,G,\alpha)$ with $A$ nuclear and $G$ exact is amenable if and only if $A$ has the $\LO$-WEP if and only if the reduced crossed product $G\ltimes A$ has the $A(G)$-WEP.
\end{abstract}

\begin{spacing}{1.0}

\maketitle


\section{Introduction}

A connection of central importance in abstract harmonic analysis is that between properties of locally compact groups and properties of their associated operator algebras. For instance, it is well-known that amenability of a locally compact group $G$ implies nuclearity of its reduced $C^*$-algebra $C^*_\lm(G)$ and injectivity of its von Neumann algebra $VN(G)$ \cite{Gui}, and that the converse is true in the setting of inner amenable groups \cite{AD,LP}. While the equivalence does not hold for connected groups \cite{Connes}, it was recently shown that amenability of a locally compact group $G$ is equivalent to injectivity of $VN(G)$ as an operator \textit{module} over the Fourier algebra $A(G)$ \cite{C}. Thus, to fully capture properties of $G$ one should not only consider the operator algebra structure of $VN(G)$, but also its operator $A(G)$-module structure. 

This perspective suggests the following question at the $C^*$-level: can one capture amenability of a locally compact group $G$ through the $A(G)$-module structure of its reduced group $C^*$-algebra $C^*_\lm(G)$? Motivated by this question, we introduce an equivariant weak expectation property (WEP) at the level of operator modules over a completely contractive Banach algebra. We answer the above question by showing that $G$ is amenable if and only if $C^*_\lm(G)$ has the $A(G)$-WEP, thereby giving a module version of a well-known result of Lance \cite{L}.

From a dynamical systems perspective, $VN(G)=\mathbb{C} \bar{\ltimes} G$ and $C^*_\lm(G)=\mathbb{C}\ltimes G$, where $G$ acts trivially on $\mathbb{C}$. If $G$ acts non-trivially on a von Neumann algebra or $C^*$-algebra, then the dual co-action induces a canonical operator $A(G)$-module structure on the respective crossed products, and it is natural to investigate whether amenability of the action can be recovered from this additional structure. In this paper we establish this for a large class of $W^*$- and $C^*$-dynamical systems over locally compact groups. Specifically, we prove that a $W^*$-dynamical system $(M,G,\alpha)$ with $M$ injective is amenable if and only if the crossed product $G\bar{\ltimes}M$ is $A(G)$-injective (Theorem \ref{t:A(G)inj}), and that a $C^*$-dynamical system $(A,G,\alpha)$ with $A$ nuclear and $G$ exact is amenable (in the sense of \cite{BEW3}) if and only if $G\ltimes A$ has the $A(G)$-WEP (see Theorem \ref{t:A(G)WEP}). 

A notion of $G$-WEP for $C^*$-dynamical systems was defined in recent work of Buss, Echterhoff and Willett \cite{BEW}. We show that this notion coincides with our $\LO$-WEP, and we perform a detailed study of the resulting $\Gamma$-WEP for $C^*$-dynamical systems $(A,\Gamma,\alpha)$ over a discrete group $\Gamma$. Among other things, we show that $ A$ has the $\Gamma$-WEP if and only if $ A^{**}$ contains a $\Gamma$-equivariant copy of the $\Gamma$-injective envelope $\mathcal{I}_{\Gamma}( A)$, and that the equivariant analogue of the QWEP conjecture fails: there is a $\Gamma$-$C^*$-algebra $B$ that is not a $\Gamma$-quotient of any $ A$ which has the $\Gamma$-WEP. We also introduce a notion of WEP for actions $\Gamma\acts A$, which, together with Lance's WEP of $ A$ entails the $\Gamma$-WEP of $ A$.

As applications of our techniques, we show that the continuous $G$-WEP as defined in \cite{BEW3} coincides with the $G$-WEP (Proposition \ref{p:4.5}), we generalize a homological characterization of amenable commutative $W^*$-dynamical systems \cite{Mon} to the non-commutative context (Proposition \ref{p:Zimmer}), and we generalize a hereditary property of amenability \cite{ADI} from discrete to arbitrary $W^*$-dynamical systems (Corollary \ref{c:5.3}).

The structure of the paper is as follows. After a preliminary section on operator modules and dynamical systems, we begin in section 3 with the definition of the $A$-WEP for operator modules over completely contractive Banach algebras. We explore some basic examples and show, among other things, that a $C^*$-algebra $A$ has the $A$-WEP if and only if it has Lance's WEP. In section 4 we study the $\Gamma$-WEP for discrete $C^*$-dynamical systems, and in section 5 we pursue examples of the $A(G)$-WEP in the setting of $W^*$- and $C^*$-dynamical systems.

\section{Preliminaries}

\subsection{Operator Modules}

Let $\mathbf{Op}$ denote the category of operator spaces and completely contractive maps. Given a completely contractive Banach algebra $A$, an object $X\in\mathbf{Op}$ is a right operator $A$-module if it is a right Banach $A$-module for which the module action extends to a complete contraction $X\pten A\to X$, where $\pten$ denotes the operator space projective tensor product. We let $\modA$ denote the category of right operator $A$-modules with completely contractive module homomorphisms. Left modules are defined analogously, and the resulting category is denoted $\Amod$. We let $A_1$ denote the unitization of $A$. Note that any $X\in\modA$ is naturally a completely contractive $A_1$-module via
$$x\cdot(a,\lm):=x\cdot a+\lm x, \ \ \ x\in X, \ a\in A, \ \lm\in \mathbb{C}.$$
Given $X\in\mathbf{Op}$, the space $\mc{CB}(A_1,X)$ is a completely contractive right $A$-module via
$$\vphi\cdot a(b)=\vphi(ab), \ \ \ a\in A, \ b\in A_1, \ \varphi\in \mc{CB}(A_1,X).$$
If, in addition, $X\in\modA$, there is a canonical completely isometric $A$-morphism $j_X:X\hookrightarrow\mc{CB}(A_1,X)$ given by
$$j_X(x)(a)=x\cdot a, \ \ \ x\in X, \ a\in A_1.$$
We also write $j_X$ for the same map taking values in $\mc{CB}(A,X)$.

We say that $X\in\modA$ is \textit{faithful} if for every non-zero $x\in X$, there is $a\in A$ such that $x\cdot a\neq 0$, and we say that $X$ is \textit{essential} if $\la X\cdot A\ra=X$, where $\la\cdot\ra$ denotes the closed linear span. Given $Z\in\Amod$, the the module tensor product $X\pten_{A}Z$ is defined by 
$$X\pten_{A}Z=X\pten Z/N, \ \ \ N=\la x\cdot a\ten z-x\ten a\cdot z\mid x\in X, \ a\in A, \ z\in Z\ra.$$

An operator module $X\in\modA$ is said to be \textit{relatively injective} if there exists a morphism $\vphi_1:\mc{CB}(A_1,X)\rightarrow X$ such that $\vphi_1\circ j_X=\id_{X}$. When $X$ is faithful, this is equivalent to the existence a morphism $\vphi:\mc{CB}(A,X)\rightarrow X$ such that $\vphi\circ j_X=\id_{X}$ by the operator analogue of \cite[Proposition 1.7]{DP}.

We say that $X$ is \textit{$A$-injective} if for every $Y,Z\in\modA$, every completely isometric morphism $\kappa:Y\hookrightarrow Z$, and every morphism $\vphi:Y\rightarrow X$, there exists a morphism $\widetilde{\vphi}:Z\rightarrow X$ such that $\widetilde{\vphi}\circ\kappa=\varphi$, that is, the following diagram commutes:

\begin{equation*}
\begin{tikzcd}
Z \arrow[rd, dotted, "\widetilde{\vphi}"]\\
Y \arrow[u, hook, "\kappa"] \arrow[r, "\vphi"] &X
\end{tikzcd}
\end{equation*}

\begin{thm}\label{t:injenv}
Every $X\in \modA$ admits an injective envelope in $\modA$, denoted $\cI_A(X)$, called the \emph{$A$-injective envelope} of $X$. The following properties hold:
\begin{enumerate}
\item
If $Y\in \modA$ and $\psi: \cI_A(X)\to Y$ is an $A$-module morphism, then $\psi$ is completely isometric if its restriction to $X$ is completely isometric.
\item
If $\psi: \cI_A(X)\to \cI_A(X)$ is an $A$-module morphism whose restriction to $X$ is the identity map, then $\psi$ is the identity map.
\end{enumerate}
\end{thm}

\begin{proof} Since $X$ is an operator space we may view it inside $\BH$ for some Hilbert space $H$. Then the canonical embedding $j_X:X\hookrightarrow\mc{CB}(A_1,X)$ extends to a completely isometric $A$-morphism $X\hookrightarrow\mc{CB}(A_1,\BH)$. Since $\BH$ is an injective operator space, $\mc{CB}(A_1,\BH)$ is injective in $\modA$ (see, e.g., the proof of \cite[Proposition 2.3]{C}, which is modelled off of \cite[Lemma 1]{Ham78}). It follows that the category $\modA$ admits sufficiently many injectives, and properties (1) and (2) follow similarly from the work of Hamana \cite{Ham78,Ham85}.
\end{proof}

A key point in the previous proof is the fact that $\mc{CB}(A_1,X)$ is $A$-injective if $X$ is an injective operator space. The following refinement of this fact in the special case that $A$ has a contractive approximate identity and $X$ is a dual operator space will be used several times later.

\begin{lem} \label{l:CB(A,M) A-inj for dual M}
If $A$ has a contractive approximate identity and $M$ is an injective dual operator space, then $\mc{CB}(A,M)$ is $A$-injective.
\end{lem}

\begin{proof}
Let $\kappa: X \hookrightarrow Y$ be an inclusion in $\modA$, and let $\varphi: X \to \mc{CB}(A,M)$ be an $A$-morphism. Let $(e_\lambda)$ be a cai for $A$. Define $\varphi_\lambda: X \to M$, $\varphi_\lambda(x) = \varphi(x)(e_\lambda)$. Then $(\varphi_\lambda)$ is a net of complete contractions in $\mc{CB}(X,M)$. Let $\varphi_0$ be a limit point of $(\varphi_\lambda)$ in the weak*-topology in $\mc{CB}(X,M)$. Since $M$ is injective, there exists a complete contraction $\tilde{\varphi}_0: Y \to M$ such that $\tilde{\varphi}_0 \circ \kappa = \varphi_0$. Define $\tilde{\varphi}: Y \to \mc{CB}(A,M)$, $\tilde{\varphi}(y)(a) = \tilde{\varphi}_0(ya)$. It is straightforward to check that $\tilde{\varphi}$ is an $A$-morphism such that $\tilde{\varphi} \circ \kappa = \varphi$.
\end{proof}

Let $ A$ be a $C^*$-algebra. A non-degenerate representation $\pi: A\rightarrow\BH$ has the weak expectation property (WEP) if there exists a unital completely positive map $\vphi:\BH\rightarrow\pi( A)''$ such that $\vphi(\pi(a))=\pi(a)$, $a\in A$. The $C^*$-algebra $ A$ has Lance's WEP if every faithful non-degenerate representation has the WEP \cite[Definition 2.8]{L}.

An object $X\in\mathbf{Op}$ has the operator space weak expectation property if for any inclusion $X\hookrightarrow Y$ in $\mathbf{Op}$ there exists a completely contractive map $\vphi:Y\rightarrow X^{**}$ such that $\vphi(x)=x$ for all $x\in X$. It was shown implicitly in \cite[pg. 459]{K} that a $C^*$-algebra $ A$ has the operator space WEP if and only if it has Lance's WEP. The underlying reason is that weak expectations are automatically completely positive:

\begin{lem}\label{l:multiplier} Let $ A\subseteq B$ be an inclusion of $C^*$-algebras. Any contraction $E:B\rightarrow A^{**}$ satisfying $E(a)=a$ for all $a\in A$ is a completely positive $A$-bimodule contraction. If, in addition, $ M( A)\subseteq B$, then $E$ is an $ M( A)$-bimodule map.
\end{lem}

\begin{proof} We follow \cite[pg. 459]{K}. The map $(E^*|_{A^*})^*:  B^{**}\rightarrow A^{**}$ satisfies $(E^*|_{ A^*})^*\circ i^{**}=\id_{ A^{**}}$, where $i$ is the inclusion $ A\subseteq  B$. Then $i^{**}\circ(E^*|_{ A^*})^*:  B^{**}\rightarrow  B^{**}$ is a projection of norm one onto $i^{**}( A^{**})$, and therefore a completely positive $i^{**}( A^{**})$-bimodule map by Tomiyama's theorem \cite{T}. Since $i$ is an injective $*$-homomorphism one sees that $(E^*|_{ A^*})^*$ is a $i^{**}( A^{**})- A^{**}$-bimodule map. It follows that $E=(E^*|_{ A^*})^*|_{  B}$ is an $ A$-bimodule map, and a $M( A)$-bimodule map in the case when $M( A)\subseteq  B$.
\end{proof}

\subsection{Locally compact groups and dynamical systems}

Let $G$ be a locally compact group. The set of coefficient functions of the left regular representation,
\begin{equation*}A(G)=\{u:G\rightarrow\mathbb{C} : u (s)=\la\lm(s)\xi,\eta\ra, \ \xi,\eta\in\LT, \ s\in G\},\end{equation*}
is called the \textit{Fourier algebra} of $G$. It was shown by Eymard that, endowed with the norm
$$\norm{u}_{A(G)}=\text{inf}\{\norm{\xi}_{\LT}\norm{\eta}_{\LT} : u(\cdot)=\la\lm(\cdot)\xi,\eta\ra\},$$
$A(G)$ is a Banach algebra under pointwise multiplication \cite[Proposition 3.4]{E}. Furthermore, it is the predual of the group von Neumann algebra $VN(G)$, where the duality is given by
\begin{equation*}\la u,\lm(s)\ra=u(s),\ \ \ u\in A(G), \ s\in G.\end{equation*}
Eymard also showed that the space of functions $\vphi:G\rightarrow\mathbb{C}$ for which there exists a strongly continuous unitary representation $\pi:G\rightarrow\mc{B}(H_\pi)$ and $\xi,\eta\in H_\pi$ such that $\vphi(s)=\la\pi(s)\xi,\eta\ra$, $s\in G$, is a unital Banach algebra (with pointwise multiplication) under the norm $$\norm{\vphi}_{B(G)}=\text{inf}\{\norm{\xi}_{H_\pi}\norm{\eta}_{H_\pi} : \vphi(\cdot)=\la\pi(\cdot)\xi,\eta\ra\},$$ called the  \textit{Fourier-Stieltjes algebra} of $G$ \cite[Proposition 2.16]{E}, denoted by $B(G)$. It is known that $B(G)$ is isometrically isomorphic to the dual of the full group $C^*$-algebra $C^*(G)$. Under this identification, states on $C^*(G)$ correspond to positive definite functions of norm one on $G$.

A $W^*$-dynamical system $(M,G,\alpha)$ consists of a von Neumann algebra $M$ endowed with an action $\alpha:G\rightarrow\mathrm{Aut}(M)$ of a locally compact group $G$ such that for each $x\in M$, the map $G\ni s\mapsto \alpha_s(x)$ is weak* continuous. Every action induces a normal $G$-equivariant injective $*$-homomorphism $\alpha:M\rightarrow\LI\oten M$ via
$$\la \alpha(x),F \ra= \int_G \la \alpha_{s^{-1}}(x),F(s) \ra \, ds, \ \ \ F \in L^1(G,M_*) = (L^\infty(G) \overline{\ten} M)_*$$
and a corresponding right $\LO$-module structure on $M$ \cite[18.6]{Str}. 
Note that the predual $M_*$ becomes a left operator $\LO$-module via $\alpha_*:\LO\pten M_*\rightarrow M_*$.

The crossed product of $M$ by $G$, denoted $G\bar{\ltimes}M$, is the von Neumann subalgebra of $\BLT\oten M$ generated by $\alpha(M)$ and $VN(G)\ten 1$. 

A $C^*$-dynamical system $( A,G,\alpha)$ consists of a $C^*$-algebra endowed with a continuous group action $\alpha:G\rightarrow\mathrm{Aut}( A)$ such that for each $a\in A$, the map $G\ni s\mapsto\alpha_s(a)\in A$ is norm continuous.

A covariant representation $(\pi, \sigma)$ of $( A,G,\alpha)$ consists of a representation $\pi: A\rightarrow\BH$ and a unitary representation $\sigma:G\rightarrow\BH$ such that $\pi(\alpha_s(a))=\sigma(s)\pi(a)\sigma(s)^{-1}$ for all $s\in G$. Given a covariant representation $(\pi,\sigma)$, we let
$$(\pi \times \sigma)(f) = \int_G \pi(f(t)) \sigma(t) \, dt, \ \ \ f\in C_c(G, A).$$
The full crossed product $G\ltimes_f  A$ is the completion of $C_c(G, A)$ in the norm
$$\|f \| = \sup_{(\pi, \sigma)} \| (\pi \times \sigma)(f)\|,$$
where $\sup$ is taken over all covariant representations $(\pi, \sigma)$ of $( A,G,\alpha)$.

Let $A\subseteq\mc{B}(H)$ be a faithful non-degenerate representation of $ A$. Then $(\alpha,\lm\ten 1)$ is a covariant representation on $L^2(G,H)$, where 
$$\alpha(a)\xi(t)=\alpha_{t^{-1}}(a)\xi(t), \ \ \ (\lm\ten 1)(s)\xi(t)=\xi(s^{-1}t), \ \ \ \xi\in L^2(G,H).$$
The reduced crossed product $G\ltimes A$ is defined to be the norm closure of $(\alpha\times(\lm\ten 1))(C_c(G,A))$. This definition is independent of the faithful non-degenerate representation $A\subseteq\mc{B}(H)$. We often abbreviate $\alpha\times(\lm\ten 1)$ as $\alpha\times\lm$. 

Analogous to the group setting, dual spaces of crossed products can be identified with certain $ A^*$-valued functions on $G$. We review aspects of this theory below and refer the reader to \cite[Chapters 7.6, 7.7]{Ped} for details. 

For each $C^*$-dynamical system $( A,G,\alpha)$ there is a universal covariant representation $(\pi,\sigma)$ such that 
$$G\ltimes_f A\subseteq C^*(\pi( A)\cup \sigma(G))\subseteq M(G\ltimes_f A).$$
Each functional $\vphi\in (G\ltimes_f A)^*$ then defines a function $\Phi:G\rightarrow A^*$ by
$$\la\Phi(s),a\ra=\vphi(\pi(a)\sigma(s)), \ \ \ a\in A, \ s\in G.$$
Let $B(G\ltimes_f A)$ denote the resulting space of $ A^*$-valued functions on $G$. An element $\Phi\in B(G\ltimes_f A)$ is \textit{positive definite} if it arises from a positive linear functional $\vphi$ as above. We let $A(G\ltimes_f A)$ denote the subspace of $B(G\ltimes_f A)$ whose associated functionals $\vphi$ are of the form
$$\vphi(x)=\sum_{n=1}^\infty\la\xi_n, \alpha\times\lm(x)\eta_n\ra, \ \ \ x\in G\ltimes_f  A,$$
for sequences $(\xi_n)$ and $(\eta_n)$ in $L^2(G,H)$ with $\sum_{n=1}^\infty \norm{\xi_n}^2<\infty$ and $\sum_{n=1}^\infty\norm{\eta_n}^2<\infty$. Then $A(G\ltimes_f A)$ is a norm closed subspace of $(G\ltimes_f A)^*$ which can be identified with $((G\ltimes A)'')_*$.

A function $h:G\rightarrow A$ is of positive type (with respect to $\alpha$) if for every $n\in\mathbb{N}$, and $s_1,...,s_n\in G$, the matrix
$$[\alpha_{s_i}(h(s_{i}^{-1}s_j)]\in M_n(A)^+.$$

Every $C^*$-dynamical system $(A,G,\alpha)$ admits a unique universal $W^*$-dynamical system $(A_\alpha'',G,\overline{\alpha})$ \cite{I}. We review this construction taking an $\LO$-module perspective. In \cite{BEW3}, they study $(A_\alpha'',G,\overline{\alpha})$ from a different, equivalent perspective.

First, $A$ becomes a right operator $\LO$-module in the canonical fashion by slicing the corresponding non-degenerate representation
$$\alpha:A\ni a \mapsto (s\mapsto\alpha_{s^{-1}}(a))\in C_b(G,A)\subseteq\LI\oten A^{**}.$$
Explicitly, this action is given by
\begin{equation}\label{e:normconv} a \ast f = \int_G  f(s)\alpha_{s^{-1}}(a) \, ds\end{equation}
for $a \in A, \ f \in L^1(G)$, where the integral is norm convergent. By duality we obtain a left operator $\LO$-module structure on $A^*$ via
$$\alpha^*|_{\LO\opten A^*}:\LO\opten A^*\rightarrow A^*.$$
Then $G$ acts in a norm-continuous fashion on the essential submodule $A^*_c:=\la\LO\ast A^*\ra$. The same argument in \cite[Lemma 7.5.1]{Ped} shows that $A^*_c$ coincides with the norm-continuous part of $A^*$, hence the notation. This fact was also noted by Hamana in \cite[Proposition 3.4(i)]{Ham11}. We therefore obtain a point-weak* continuous action of $G$ on the dual space $(A^*_{c})^*$ by surjective isometries. Clearly 
\begin{equation}\label{e:iden}(A^*_{c})^*\cong A^{**}/(A^*_{c})^{\perp}\end{equation}
completely isometrically and weak*-weak* homeomorphically as right $\LO$-modules, where the canonical $\LO$-module structure on $A^{**}$ is obtained by slicing the normal cover of $\alpha$, which is the normal $*$-homomorphism
$$\widetilde{\alpha}=(\alpha^*|_{\LO\opten A^*})^*:A^{**}\rightarrow\LI\oten A^{**}.$$
Note that $\widetilde{\alpha}|_{M(A))}$ is the unique strict extension of $\alpha$, and is therefore injective \cite[Proposition 2.1]{Lance}. However, on $A^{**}$, $\widetilde{\alpha}$ can have a large kernel. On the one hand, its kernel is of the form $(1-z)A^{**}$ for some projection $z\in Z(A^{**})$. On the other hand, by definition of the $\LO$-action on $A^{**}$, $\mathrm{Ker}(\widetilde{\alpha})=(A^*_{c})^{\perp}$. It follows that $(A^*_{c})^*$ is completely isometrically weak*-weak* order isomorphic to $zA^{**}$, where we equip $(A^*_{c})^*$ with the quotient operator system structure from $A^{**}$. We can therefore transport the point-weak* continuous $G$ action on $(A^*_{c})^*$ to $A_\alpha'':=zA^{**}$, yielding a $W^*$-dynamical system $(A_\alpha'',G,\overline{\alpha})$, where $\overline{\alpha}:G\rightarrow\mathrm{Aut}(A_\alpha'')$ is given by
$$\overline{\alpha}_t(zx)=z((\alpha_t)^{**}(x)), \ \ \ x\in A^{**}, \  t\in G.$$
We emphasize that with this structure $A_\alpha''$ is not necessarily an $\LO$-submodule of $A^{**}$, rather $\mathrm{Ad}(z):A^{**}\rightarrow A_\alpha''$ is an $\LO$-quotient map. 

\section{The weak expectation property for operator modules}

Throughout this section, unless otherwise stated, $A$ denotes a fixed completely contractive Banach algebra. For a Banach (or operator) space $X$, we let $i_X:X\hookrightarrow X^{**}$ denote the canonical inclusion.

\begin{defn}\label{d:AWEP} 
An object $X\in\modA$ has the \textit{weak expectation property ($A$-WEP)} if for any completely isometric morphism $\kappa: X\hookrightarrow Y$ there exists a morphism $\psi:Y\rightarrow X^{**}$ such that $\psi\circ\kappa =i_X$.\end{defn}

\begin{examples} \label{examps, incl dual mods}${}$\vskip2pt
\begin{enumerate}
\item Clearly, we recover the operator space WEP when $A=\bC$.
\item Any $A$-injective module has the $A$-WEP. If, in addition, $X\in\modA$ is a dual module in the sense that there exists a $Y\in\Amod$ with $X=Y^*$ and $\langle x a, y \rangle = \langle x,ay \rangle$ for all $x \in X$, $y \in Y$, and $a \in A$ (equivalently, $X$ has an operator space predual with respect to which $X \ni x \mapsto x \cdot a \in X$ is weak*-weak* continuous for each $a \in A$), then $X$ has the $A$-WEP if and only if $X$ is $A$-injective. This follows quickly from the fact that the adjoint of the inclusion $Y \hookrightarrow X^*$ is an $A$-module projection $X^{**} \to X$.
\end{enumerate}
\end{examples}

\begin{remark} Notions of WEP for a $C^*$-algebra $A$ relative to another $C^*$-algebra $B$ appeared in the unpublished manuscript \cite{LR}. They are defined in terms of relative weak injectivity of inclusions of the type $A\subseteq\mc{L}(E_B)$, where $E_B$ is a Hilbert $B$-module and $\mc{L}(E_B)$ is the $C^*$-algebra of adjointable operators on $E_B$. Although similar in nature, these notions differ from ours as there need not be a canonical $B$-module structure on $A$, in general. Even when $A=B=C^*_\lm(\mathbb{F}_2)$, by \cite[Example 5.4]{LR}, $A$ has the $AWEP_2$ in the sense of \cite[Definition 3.1]{LR}, but it follows from Proposition \ref{c:h2} and \cite{L} that $A$ does not have the $A$-WEP in the sense of Definition \ref{d:AWEP}. Besides, in this paper we are mainly interested in the case where the underlying Banach algebra $A$ is not a $C^*$-algebra.
\end{remark}

We record a number of equivalent conditions for later use.
 
\begin{thm}\label{A-WEP-eq-conds}
The following are equivalent for $X\in\modA$:
\begin{enumerate}
\item
$X$ has the $A$-WEP.
\item
For any inclusion $\kappa:X\hookrightarrow Y$ there exists a morphism $\vphi:X^*\rightarrow Y^*$ such that $\kappa^*\circ \vphi=\id_{X^*}$.
\item
For any inclusion $\kappa: X\hookrightarrow Z^*$ into a dual $A$-module $Z^*$, 
there exists a morphism $\psi:Z^*\rightarrow X^{**}$ such that $\psi\circ\kappa =i_X$.
\item
For any inclusion $\kappa: X\hookrightarrow Z^*$ into a dual $A$-module $Z^*$, 
there exists a morphism $\tilde\psi:Z^*\rightarrow \overline{\kappa(X)}^{\text{weak*}}$ 
such that $\tilde\psi\circ\kappa =\id_{X}$.
\item
There is an $A$-module embedding $\mathrm{i}: \cI_A(X) \hookrightarrow X^{**}$ such that $\mathrm{i}|_X =i_X$.
\item
The inclusion $i_X:X\hookrightarrow X^{**}$ factors through an injective module in $\modA$.
\item
For every inclusion $X\hookrightarrow Y$ in $\modA$ and $Z\in\Amod$ the canonical map $X\pten_{A}Z\hookrightarrow Y\pten_{A}Z$ is a complete isometry.
\end{enumerate}
\end{thm}

\begin{proof}
$(1)\implies(2)$: Suppose there exists a morphism $\psi:Y\rightarrow X^{**}$ such that $\psi\circ\kappa =i_X$. Then $\vphi:=\psi^*|_{X^*}: X^*\to Y^*$ is a morphism that satisfies 
$\langle \kappa^*\circ \vphi(x^*), x^{**} \rangle = \langle x^*, \psi^{**}\circ\kappa^{**} (x^{**}) \rangle = \langle x^*, (\psi\circ\kappa)^{**} (x^{**}) \rangle = \langle x^*, x^{**} \rangle$, which shows $\kappa^*\circ \vphi =\id_{X^*}$.

$(2)\implies(1)$: Let $\kappa:X\hookrightarrow Y$ be an inclusion. Suppose there exists a morphism $\vphi:X^*\rightarrow Y^*$ such that $\kappa^*\circ \vphi=\id_{X^*}$. Then $\psi:=\vphi^*|_Y: Y\rightarrow X^{**}$ is a morphism 
and for each $x\in X$ and $x^*\in X^*$ we have $\langle x^*, \psi(\kappa(x)) \rangle = \langle \kappa^*\circ\vphi(x^*), x \rangle = \langle x^*, x \rangle$, which shows $\psi\circ\kappa =i_X$.

$(1)\implies(3)$: Obvious.

$(3)\implies(1)$: Follows from the fact that $Y$ embeds into $Y^{**}$ as a submodule.

$(1)\implies(4)$: First, since $i_Z^*\circ\kappa^{**}|_X = \kappa$ it follows that
$i_Z^*\circ\kappa^{**}(X^{**}) \subseteq \overline{\kappa(X)}^{\text{weak*}}$.
Now let $\psi: Z^* \to X^{**}$ be a morphism with $\psi\circ\kappa =i_X$. 
Then the composition $\tilde\psi:=i_Z^*\circ\kappa^{**}\circ\psi$ is the desired morphism.

$(4)\implies(1)$: Let $\kappa:X\hookrightarrow Y$ be an inclusion. 
Let $\tilde\kappa:=\kappa^{**}\circ i_X : X\rightarrow Y^{**}$. Note that 
$\overline{\tilde\kappa(X)}^{\text{weak*}} = \kappa^{**}(X^{**})$.
Now by $(4)$ there is a morphism
$\tilde\psi: Y^{**}\rightarrow \overline{\tilde\kappa(X)}^{\text{weak*}}$ 
such that $\tilde\psi\circ\tilde\kappa =\id_{X}$.
Thus the restriction of $(\kappa^{**})^{-1}\circ\tilde\psi:Y^{**}\rightarrow X^{**}$ to $Y$ 
is the desired morphism.

$(1)\implies(5)$: Follows from the definition.

$(5)\implies(1)$: Let $\kappa:X\hookrightarrow Y$ be an inclusion. By injectivity, the map $\kappa^{-1}:\kappa(X)\to X$ extends to a morphism $\varphi: Y \to \cI_A(X)$. Then the composition ${\text i}\circ \varphi: Y \to X^{**}$ is a morphism that satisfies $\psi\circ\kappa =i_X$.

$(5)\implies(6)$: Obvious.

$(6)\implies (1)$: If there exists an injective module $I$ and a morphisms $\vphi:X\rightarrow I$ and $\psi:I\rightarrow X^{**}$ satisfying $\psi\circ\vphi=i_X$, then for any inclusion $\kappa:X\rightarrow Y$ there is a morphism $\widetilde{\vphi}:Y\rightarrow I$ with $\widetilde{\vphi}\circ\kappa=\vphi$. Then $\psi\circ\widetilde{\vphi}:Y\rightarrow X^{**}$ is the desired weak expectation.

$(2)\Longleftrightarrow (7)$: Let $\kappa:X\hookrightarrow Y$ be an inclusion. Then $(\kappa\ten\id_Z):X\pten_A Z\rightarrow Y\pten_A Z$ is a complete isometry for every $Z\in\Amod$ if and only if $\kappa$ is a weak retract, meaning there exists a morphism $\vphi:X^*\rightarrow Y^*$ satisfying $\kappa^*\circ\vphi=\id_{X^*}$. This follows verbatim from (the operator space analogue) of \cite[1.9]{CLM}.
\end{proof}

An inclusion $X\hookrightarrow Y$ in $\modA$ for which $X\pten_{A}Z\hookrightarrow Y\pten_{A}Z$ is a complete isometry for every $Z\in\Amod$ is said to be \textit{flat}.

\begin{remark} The equivalence of (1) and (7) in Theorem \ref{A-WEP-eq-conds} is the operator module analogue of Lance's characterization of the WEP for a $C^*$-algebra $ A$ by means of the so-called extension property for $\ten_{\max}$ \cite[Theorem 3.3]{L}, meaning that for any inclusion $ A\subseteq B$ of $C^*$-algebras, and any $C^*$-algebra $C$, we have the inclusion $ A\ten_{\max}C\subseteq B\ten_{\max} C$.
\end{remark}

\subsection{$A$-WEP vs. WEP}
It is natural to wonder whether the $A$-WEP implies the operator space WEP. This is false, in general, as the following example shows.

\begin{example}
Let $\Gamma$ be a non-amenable discrete group, $A = B(\Gamma)$, the Fourier-Stieltjes algebra of $\Gamma$, and $X = W^*(\Gamma) = C^*(\Gamma)^{**}$ the universal von Neumann algebra of $\Gamma$. Since $B(\Gamma)$ is unital it follows that is 1-projective over itself in the sense of \cite[Section 2]{C}, so that (by the module version of \cite[Theorem 3.5]{B}) $W^*(\Gamma) = B(\Gamma)^*$ is $B(\Gamma)$-injective, thus has the $B(\Gamma)$-WEP. However, if $W^*(\Gamma)$ had the operator space WEP, then, as it is a dual space, it would necessarily be injective by Example \ref{examps, incl dual mods} (2). This would entail nuclearity of $C^*(\Gamma)$ and hence the amenability of $\Gamma$, hence we have a contradiction.
\end{example}

The following theorem characterizes the $A$ for which the implication $A$-WEP $\Rightarrow$ WEP always holds. If $A$ is unital, then $A^*$ has the $A$-WEP. So the equivalence of (1) and (5) of this theorem says that, in the unital case, as long as $A^*$ is not a counterexample to ``$A$-WEP implies WEP,'' then there are no counterexamples.

\begin{thm} \label{AWEP implies WEP char}
For a completely contractive Banach algebra $A$, the following are equivalent:
\begin{enumerate}
\item[$(1)$] Every operator module over $A$ that is $A$-injective is injective.
\item[$(2)$] Every operator module over $A$ that has the $A$-WEP has the WEP.
\end{enumerate}

If $A$ has a cai, then the following are also equivalent to the above:
\begin{enumerate}
\item[$(3)$] For any Hilbert space $H$, $\mc{CB}(A,\cB(H))$ is injective.
\item[$(4)$] The operator space dual $A^*$ is injective.
\item[$(5)$] The operator space dual $A^*$ has the WEP.
\end{enumerate}
\end{thm}

\proof
$(1) \Longrightarrow (2)$: Suppose $X$ has the $A$-WEP, so that $\cI_A(X)$ embeds in $X^{**}$ via an embedding that restricts to the identity on $X$. By (1), there is also an embedding $\cI(X) \subseteq \cI_A(X)$ that restricts to the identity on $X$. By composing these embeddings, we see that $X$ has the WEP.

$(2) \Longrightarrow (1)$: If $X$ is $A$-injective, then $X$ has the $A$-WEP. By (2), $X$ has the WEP, and so $\cI(X) \subseteq X^{**}$ via an embedding restricting to the identity on $X$. Since $X$ is $A$-injective, there is an $A$-morphism $\Phi: X^{**} \to X$ extending the identity. Restricting $\Phi$ to a copy of $\cI(X)$ yields a complete contraction $\cI(X) \to X$ restricting to the identity on $X$. It follows that $X$ is injective.

Now assume that $A$ is unital.

$(1) \Longrightarrow (3)$: Follows from \ref{l:CB(A,M) A-inj for dual M}.

$(3) \Longrightarrow (1)$: If $X$ is $A$-injective, then there exists an $A$-morphism $\Phi: \mc{CB}(A,\cB(H)) \to X$ that restricts to the identity on the canonical copy of $X$ in $\mc{CB}(A,\cB(H))$ via $j_X$. It is then clear that if (3) holds, $X$ is injective.

$(3) \Longrightarrow (4)$: Obvious (take $H = \mathbb C$).

$(4) \Longrightarrow (3)$: If $H$ is a Hilbert space with dimension $I$, then we may canonically identify $\mc{CB}(A,\cB(H))$ with the space $\mathbb M_I(A^*)$ of matrices $[\varphi_{ij}]$ indexed by a set with cardinality $I$ with entries in $A^*$ such that the finitely supported submatrices of $[\varphi_{ij}]$ are uniformly bounded in norm. If $A^*$ is injective, then there is a Hilbert space $K$, completely isometric representation $A^* \subseteq \cB(K)$, and completely contractive projection $\Phi:\cB(K) \to A^*$ that restricts to the identity on $A^*$. The canonical amplification $\Phi_I: \cB(H \otimes^2 K) = \mathbb M_I(\cB(K)) \to \mathbb M_I(A^*) = \mc{CB}(A,\cB(H))$ is then a completely contractive projection, which implies that $\mc{CB}(A,\cB(H))$ is injective since $\cB(H \otimes^2 K)$ is.

$(4) \Longleftrightarrow (5)$: This is a simple consequence of the fact that the adjoint of the inclusion $A \hookrightarrow A^{**}$ is a conditional expectation onto $A^*$. \endproof

Say that a left $A$-module $X$ is an \emph{$h$-module} over $A$ if the module action extends to a complete contraction $A \otimes^h X \to X$, where $\otimes^h$ is the Haagerup tensor product. Since there is a canonical complete contraction $A \pten X \to A \otimes^h X$, it follows that every $h$-module over $A$ is an operator $A$-module.

The following result, and hence also the corollary below it, actually holds in general for any nondegenerate $h$-module over an approximately unital Banach algebra with an operator space structure \cite[comment above Theorem 2.2]{BP}. For convenience, we provide an elementary proof in the case that the algebra is a $C^*$-algebra.

\begin{prop} \label{X h-mod implies I(X) hmod}
If $X$ is a non-degenerate $h$-module over a $C^*$-algebra $A$, then there is an $A$-module structure on $\cI(X)$ extending that on $X$, and $\cI(X)$ is an $h$-module over $A$ with this structure.
\end{prop}

\proof
First assume $A$ is unital. Let $m: A \otimes_h X \to X$ be a complete contraction extending the module action. Since $\otimes_h$ is injective, $A \otimes_h X \subseteq A \otimes_h \cI(X)$. Thus $m$ extends to a complete contraction $\tilde m: A \otimes_h \cI(X) \to \cI(X)$. Define an action of $A$ on $\cI(X)$ by $a \cdot \eta = \tilde m(a \otimes \eta)$ for $a \in A$, $\eta \in \cI(X)$.

The only nontrivial property to check in order to prove that this is a module action is the identity $(ab) \cdot x = a \cdot (b \cdot x)$. Fix a unitary $a \in A$, and let $\varphi: \cI(X) \to \cI(X)$ be the map $\eta \mapsto a^{-1} \cdot (a \cdot \eta)$. Since $\varphi$ is a complete contraction restricting to the identity on $X$, we have $\varphi = \text{id}$ by rigidity. It follows similarly from rigidity that for unitaries $a, b \in A$, $\eta = (ab)^{-1} \cdot (a \cdot (b \cdot \eta))$. So \[(ab)\cdot \eta = (ab) \cdot ((ab)^{-1} \cdot (a \cdot (b \cdot \eta))) = a \cdot (b \cdot \eta)\] for all unitaries $a,b \in A$ and $\eta \in \cI(X)$. Since the unitaries in $A$ span $A$, the desired identity holds for general $a,b \in A$. Thus $\cI(X)$ admits an $A$-module structure extending that on $X$ for which $\cI(X)$ becomes an $h$-module over $A$. 

If $A$ is non-unital, then $X$ is canonically an $h$-module over the unitization $A_1$ \cite[3.1.11]{BLM}. So by the first part of the present proof, there is some $A_1$-module structure on $\cI(X)$ extending that on $X$ for which $\cI(X)$ is a $h$-module over $A_1$. Restricting to $A$ and using functoriality of the Haagerup tensor product gives the desired $A$-module structure on $\cI(X)$. 
\endproof

\begin{cor}\label{c:h}
If $X$ is a nondegenerate $h$-module over a $C^*$-algebra $A$, and $X$ has the $A$-WEP, then $X$ has the WEP. In particular, every $C^*$-algebra $A$ which has the $A$-WEP has the WEP.
\end{cor}

\proof
By Proposition \ref{X h-mod implies I(X) hmod}, $\cI(X)$ is an $h$-module over $A$. Hence $\cI(X) \in \textbf{Amod}$ by the comment above Proposition \ref{X h-mod implies I(X) hmod}. So if $X$ has the $A$-WEP, then there is a completely contractive ($A$-module) map $\psi: \cI(X) \to X^{**}$ that restricts to the inclusion on $X$.
\endproof

\begin{cor}\label{c:h2} A $C^*$-algebra $A$ has the $A$-WEP if and only if it has the WEP.
\end{cor}

\begin{proof} One direction follows immediately from Corollary \ref{c:h}. Suppose $A$ has the WEP. Then the inclusion $A\hookrightarrow A^{**}$ factors through the injective envelope $\mc{I}(A)$, say through a completely positive contraction $E:\mc{I}(A)\rightarrow A^{**}$. Since $E$ is a weak expectation, it is an $A$-bimodule map by Lemma \ref{l:multiplier}. Moreover, $\mc{I}(A)$ an injective operator $A$-module (see \cite[pg. 60]{FP}). This follows from Wittstock's bimodule extension theorem \cite[Theorem 4.1]{Witt2} and Tomiyama's theorem \cite{T} on conditional expectations: any faithful inclusion $A\subseteq\BH$ lifts to a complete contraction $\vphi:\mc{I}(A)\rightarrow\BH$, which is automatically a complete isometry by rigidity. Thus, injectivity of $\mc{I}(A)$ and Tomiyama's theorem yield a completely contractive $A$-bimodule projection $P:\BH\rightarrow \cI(A)$. Since $\BH$ is $A$-injective by \cite[Theorem 4.1]{Witt2}, it follows that $\cI(A)$ is $A$-injective. Thus, the inclusion $A\hookrightarrow A^{**}$ factors through an $A$-injective module, implying $A$ has the $A$-WEP by Theorem \ref{A-WEP-eq-conds} (6).
\end{proof}

\begin{remark} Corollary \ref{c:h2} is the WEP analogue of \cite[Theorem 3.2]{FP} which states (in particular) that a unital $C^*$-algebra $A$ is injective if and only if it is $A$-injective.
\end{remark}

\subsection{The $A$-module $\mathbb C$}

Since a completely contractive $A$-module structure on an operator space $X$ is equivalent to a complete contraction $A \to \mc{CB}(X)$, there is a one-to-one correspondence between characters (i.e., multiplicative linear functionals) on $A$ and completely contractive $A$-module structures on $\mathbb C$. For a character $\varphi$ on $A$, denote by $\mathbb C_\varphi$ the space $\mathbb C$ with the corresponding $A$-module structure, i.e., $a \cdot z = \varphi(a) z$ for $a \in A$, $z \in \mathbb C$. It is natural to ask for conditions on $\varphi$ that are equivalent to the $A$-WEP of $\mathbb C_\varphi$, which by Example \ref{examps, incl dual mods} (2) is equivalent to $A$-injectivity of $\mathbb C_\varphi$.

The characterization below uses the notion of \emph{$\varphi$-amenability} due to Kaniuth, Lau, and Pym \cite{KLP}. For a character $\varphi$ on $A$, $A$ is said to be $\varphi$-amenable if there exists a bounded linear functional $m$ on $A^*$ such that $\langle m, \varphi \rangle = 1$ and $\langle m, f \cdot a \rangle = \varphi(a) \langle m, f \rangle$ for all $a \in A$ and $f \in A^*$. Such a functional $m$ is called a \emph{$\varphi$-mean}.

\begin{prop}
For a character $\varphi$ on $A$, if $\mathbb C_\varphi$ is $A$-injective, then $A$ is $\varphi$-amenable with a $\varphi$-mean of norm one. The converse holds if $A$ has a cai.
\end{prop}

\proof
Suppose that $\mathbb C_\varphi$ is $A$-injective. Define a map $i: \mathbb C_\varphi \to A^*$ by $i(z) = z \varphi$. Then $i$ is completely contractive since $\varphi$ is contractive, and
\[ \langle i(a \cdot z), b \rangle = \varphi(a)z\varphi(b) = z \langle \varphi, ba \rangle = \langle a i(z), b \rangle \] for all $a, b \in A$, $z \in \mathbb C$. So $i$ is an $A$-morphism. By assumption, there is an $A$-morphism $m: A^* \to \mathbb C$ such that $m \circ i  = \text{id}_\mathbb C$. Then $\|m\| \leq 1$, $\langle m, \varphi \rangle = \langle m,i(1) \rangle = 1$, and for any $f \in A^*$ and $a \in A$, $\langle m, fa \rangle = a \langle m,f \rangle = \varphi(a) \langle m,f \rangle$.

For the converse, assume that $A$ has a cai and that $A$ is $\varphi$-amenable with a $\varphi$-mean of norm one. Let $i: \mathbb C_\varphi \to A^*$ be as above. By similar calculations to those above, any norm-one $\varphi$-mean $m: A^* \to \mathbb C_\varphi$ is an $A$-morphism such that $m \circ i = \text{id}_\mathbb C$. Then $\mathbb C_\varphi$ is $A$-injective since, by Lemma \ref{l:CB(A,M) A-inj for dual M}, $A^*$ is $A$-injective.
\endproof

\section{The $G$-WEP}

In \cite{BEW}, Buss, Echterhoff and Willett introduced a notion of $G$-WEP for $C^*$-dynamical systems over locally compact groups $G$: a $G$-$C^*$-algebra $A$ has the $G$-WEP if for any $G$-equivariant inclusion $A\subseteq B$ into another $G$-$C^*$-algebra $B$, there exists a $G$-equivariant completely positive contraction $E:B\rightarrow A^{**}$ which restricts to the identity on $A$. In this subsection, we show that this notion coincides with the $L^1(G)$-WEP. Useful tools in this regard are the following two lemmas.

\begin{lem}\label{l:flat}
Let $A$ be a completely contractive Banach algebra with a (two-sided) contractive approximate identity, and let $X\in\modA$. Then there exists a completely contractive $A$-module map $\varphi: X \to \la X\cdot A\ra^{**}$ such that $\varphi|_{\la X\cdot A\ra} = i_{\la X\cdot A\ra}$. Hence $\la X\cdot A\ra\hookrightarrow X$ is a flat inclusion.
\end{lem}

\begin{proof}
Let $(e_\alpha)$ be a (two-sided) contractive approximate identity for $ A$. Then for each $\alpha$, 
$$R_{e_{\alpha}}:X\ni x\mapsto x\cdot e_{\alpha}\in \la X\cdot A\ra$$
is completely contractive, so the composition $\varphi_\alpha := i_{\la X\cdot A\ra} \circ R_{e_{\alpha}}: X \to \la X\cdot A\ra^{**}$ is completely contractive. Let $\varphi$ be a limit point of $(\varphi_\alpha)$ in the weak*-topology of $\mc{CB}(X,\la X \cdot A \ra^{**})$. Then $\varphi$ is completely contractive, and it is straightforward to check that $\varphi$ is an $A$-module map with $\varphi|_{\la X\cdot A\ra} = i_{\la X\cdot A\ra}$.

The last statement follows as in the proof of $(2)\Leftrightarrow(7)$ in Theorem \ref{A-WEP-eq-conds}.
\end{proof}

\begin{lem}\label{l:essAWEP}
Let $A$ be a completely contractive Banach algebra with a (two-sided) contractive approximate identity, and let $X\in\modA$ be essential (i.e., $\la X \cdot A \ra = X$). Then $X$ has the $A$-WEP if and only if for every completely isometric morphism $\kappa: X \hookrightarrow Y$ into an essential $A$-module $Y$, there exists a morphism $\varphi: Y \to X^{**}$ such that $\varphi \circ \kappa = i_X$. 
\end{lem}

\begin{proof}
The forward direction is obvious. For the converse, let $\kappa: X \hookrightarrow Y$ be an inclusion of $A$-modules. Then $\kappa(X) \subseteq \la Y \cdot A \ra$ since $X$ is essential. So there exists a morphism $\psi: \la Y \cdot A \ra \to X^{**}$ such that $\varphi \circ \kappa = i_X$, and $\psi$ extends to a morphism $\la Y \cdot A \ra^{**} \to X^{**}$. Composing this map with the morphism $Y \to \la Y \cdot A \ra^{**}$ guaranteed by Lemma \ref{l:flat} yields the desired map $Y \to X^{**}$.
\end{proof}

Let $(M,G,\alpha)$ be a $W^*$-dynamical system. Then $M$ is canonically a module over $M(G)$ via the following weak*-convergent integral:
\begin{equation}\label{e:int} x \ast \mu=\int_G \alpha_{s^{-1}}(x) \ d\mu(s), \ \ \ x\in M, \ \mu\in M(G).\end{equation}
The continuous part of $M$,
\[ M^c := \{x \ast f : f \in L^1(G), x \in M\},\]
coincides with the set $\{x \in M : t \mapsto \alpha_t(x) \text{ is norm-continuous}\}$, and is a $G$-$C^*$-algebra. Moreover, for each $x\in M^c$, the integral (\ref{e:int}) is norm convergent. $M^c$ is a module over $M(G)$ and an essential module in $\mathbf{mod}L^1(G)$. These facts are well-known (see \cite{Ped} or \cite{Ham11}, for example). 

\begin{prop}\label{p:G-WEP} Let $(A,G,\alpha)$ be a $C^*$-dynamical system. Then $A$ has the $G$-WEP if and only if it has the $\LO$-WEP.
\end{prop}

\begin{proof} Suppose $A$ has the $\LO$-WEP and that $A\subseteq B$ is an inclusion of $G$-$C^*$-algebras. Then there exists an $\LO$-morphism $E:B\rightarrow A^{**}$ restricting to the identity on $A$. For any $f\in\LO$, $s\in G$, and $b\in B$ we have
$$E(\alpha_{s}(b\ast f))=E((b\ast f)\ast\delta_s)=E(b\ast(f\ast\delta_s))=E(b)\ast(f\ast\delta_s)=\alpha_s(E(b)\ast f)= \alpha_s(E(b \ast f)).$$
Since $B$ satisfies $\la B\ast\LO\ra = B$, it follows that $E$ is a $G$-equivariant complete contraction. By Lemma \ref{l:multiplier} $E$ is necessarily a completely positive contraction, whence $A$ has the $G$-WEP.

Conversely, suppose that $A$ has the $G$-WEP. Let $Y\in\mathbf{mod}\LO$ and $\kappa:A\hookrightarrow Y$ be a completely isometric $\LO$-morphism. We will show that there exists an $L^1(G)$-morphism $\varphi: Y \to A^{**}$ such that $\varphi \circ \kappa = i_A$. By Lemma \ref{l:essAWEP}, we may assume $Y$ is essential. Let $A \subseteq \BH$ be a faithful $*$-representation. Since $A \in \mathbf{mod}\LO$ is essential and $L^1(G)$ has a contractive approximate identity, it follows that the canonical $*$-homomorphic $L^1(G)$-morphism $j: A \to \mc{CB}(L^1(G),\BH)$ is completely isometric. By Lemma \ref{l:CB(A,M) A-inj for dual M}, $\mc{CB}(L^1(G),\BH)$ is $L^1(G)$-injective. So there exists an $L^1(G)$-morphism $\psi: Y \to \mc{CB}(L^1(G),\BH)$ such that $\psi \circ \kappa = j$. Since $Y$ is essential, $\psi(Y) \subseteq B$, where $B = \la \mc{CB}(L^1(G),\BH) \cdot L^1(G) \ra$. Since $\mc{CB}(L^1(G),\BH) = \BH \overline{\otimes} L^\infty(G)$ is a $G$-$W^*$-algebra, $B$ is a $G$-$C^*$-algebra containing a $*$-homomorphic copy of $A$. So there is a $G$-equivariant complete contraction, which must also be an $L^1(G)$-morphism, $\rho: B \to A^{**}$ such that $\rho \circ j = i_A$. Then $\varphi := \rho \circ \psi: Y \to A^{**}$ is an $L^1(G)$-morphism, and $\varphi \circ \kappa = \rho \circ \psi \circ \kappa = \rho \circ j = i_A$.
\end{proof}

In a similar fashion, we have the analogous result for injectivity. Here, we say that $G$-$C^*$-algebra is $G$-injective if it is injective in the category of $G$-$C^*$-algebras and completely positive $G$-equivariant contractions, denoted $\mathbf{G}$-$\mathbf{C^*}$-$\mathbf{alg}$.

\begin{prop}\label{p:G-inj} Let $(A,G,\alpha)$ be a $C^*$-dynamical system. Then $A$ is $G$-injective if and only if it is $\LO$-injective.
\end{prop} 

\begin{proof} Suppose $A$ has is injective in $\mathbf{mod}\LO$ and that $B\subseteq C$ is an inclusion of $G$-$C^*$-algebras. If $\vphi:B\rightarrow A$ is a morphism in $\mathbf{G}$-$\mathbf{C^*}$-$\mathbf{alg}$, then by norm convergence in the $C^*$-analogue of (\ref{e:int}), $\vphi$ is an $\LO$-module map. By $\LO$-injectivity there exists an $\LO$-morphism $\widetilde{\vphi}:C\rightarrow A$ extending $\vphi$. By Lemma \ref{l:multiplier} $\widetilde{\vphi}$ is necessarily a completely positive contraction, and as above it is $G$-equivariant. Whence $A$ is $G$-injective.

Conversely, if $A$ is $G$-injective, and $\iota:X\hookrightarrow Y$ is an inclusion in $\mathbf{mod}\LO$, and $\vphi:X\rightarrow A$ is an $\LO$-morphism. Let $A\subseteq\BH$ be a faithful inclusion and $j_A:A\hookrightarrow\mc{CB}(\LO_1,\BH)$ the canonical embedding. By \cite[Proposition 2.3]{C} (which is inspired by \cite[Lemma 2.2]{Ham85}) $\mc{CB}(\LO_1,\BH)$ is $\LO$-injective. Thus, there exists an $\LO$-morphism $\widetilde{\vphi}:Y\rightarrow\mc{CB}(\LO_1,\BH)$ extending $j_A\circ\vphi$. Since 
$$\mc{CB}(\LO_1,\BH)=(\LO_1\pten\Th)^*=(\LI\oplus\mathbb{C}1)\oten\BH$$
is a $G$-$W^*$-algebra, its continuous part $\mc{CB}(\LO_1,\BH)^c$ is a $G$-$C^*$-algebra. Moreover, $j_A(A)\subseteq \mc{CB}(\LO_1,\BH)^c$. By $G$-injectivity there is a $G$-equivariant completely positive contraction $\Phi:\mc{CB}(\LO_1,\BH)^c\rightarrow A$ satisfying $\Phi\circ j_A=\id_A$. As above, $\Phi$ is necessarily an $\LO$-morphism, and the composition $\Phi\circ\widetilde{\vphi}:Y\rightarrow A$ is the desired $\LO$-morphism extending $\vphi$.

\end{proof}

In \cite[Definition 8.1]{BEW3}, Buss, Echterhoff and Willett introduced a notion of \textit{continuous $G$-WEP} for a $C^*$-dynamical system $(A,G,\alpha)$: for any $G$-equivariant inclusion $A\subseteq B$ into a $G$-$C^*$-algebra $B$, there exists a completely positive $G$-equivariant contraction $E:B\rightarrow A_\alpha''$ for which $E|_A$ is the canonical inclusion $A\subseteq A_\alpha''$. By \cite[Lemma 8.1]{BEW3}, the $G$-WEP implies the continuous $G$-WEP. We now establish the converse.

\begin{prop}\label{p:4.5} Let $(A,G,\alpha)$ be a $C^*$-dynamical system. Then $A$ has the $G$-WEP if and only if it has the continuous $G$-WEP.
\end{prop}

\begin{proof} First, we construct a unital completely positive $\LO$-module map $A_\alpha''\rightarrow A^{**}$ which restricts to the identity on $A$. Let $\widetilde{\alpha}:A^{**}\rightarrow\LI\oten A^{**}$ be the normal cover of the non-degenerate representation $\alpha:A\rightarrow M(C_0(G)\iten A)\subseteq\LI\oten A^{**}$, let $(f_i)$ be a contractive approximate identity for $\LO$ consisting of states, and let $m\in\LI^*$ be a weak* limit of a subnet $(f_{i_j})$. The map
$$\Phi:A^{**}\ni a\mapsto (m\ten\id)\widetilde{\alpha}(a)\in A^{**},$$
is a unital completely positive $\LO$-morphism, the latter property following from the asymptotic centrality of $(f_{i_j})$: for every $a\in A^{**}$, $f\in\LO$ and $\mu\in A^*$,
\begin{align*}\la\Phi(a\ast f),\mu\ra&=\lim_{j}\la a\ast f,f_{i_j}\ast\mu\ra=\lim_{j}\la a,f\ast f_{i_j}\ast\mu\ra=\lim_{j}\la a,f_{i_j}\ast f\ast\mu\ra\\
&=\lim_{j}\la a\ast f_{i_j},f\ast\mu\ra=\la\Phi(a),f\ast\mu\ra\\
&=\la\Phi(a)\ast f,\mu\ra.
\end{align*}
By definition of $\Phi$, it follows that $(A^*_{c})^{\perp}=(1-z)A^{**}\subseteq\mathrm{Ker}(\Phi)$, so we obtain an induced unital completely positive $\LO$-morphism (still denoted) $\Phi:A_\alpha''\rightarrow A^{**}$. Then for every $a\in A$ and $\mu\in A^*$,
$$\la\Phi(za),\mu\ra=\la(m\ten\id)\widetilde{\alpha}(za),\mu\ra=\la(m\ten\id)\alpha(a),\mu\ra=\lim_{j}\la a\ast f_{i_j},\mu\ra=\la a,\mu\ra,$$
where the last equality follows from continuity of $s\mapsto \alpha_s(a)$. 

Now, suppose $A$ has the continuous $G$-WEP and $B$ is a $G$-$C^*$-algebra for which $A\subseteq B$. Then there exists a completely positive $G$-equivariant contraction $E:B\rightarrow A_\alpha''$ which restricts to the inclusion $A\subseteq A_\alpha''$. By norm convergence of (\ref{e:normconv}) it follows that $E$ is $\LO$-equivariant. Then $\Phi\circ E:B\rightarrow A^{**}$ is a completely positive $\LO$-morphism which restricts to the identity on $A$. By $\LO$-essentiality of $B$, the same argument from the proof of Proposition \ref{p:G-WEP} shows that $\Phi\circ E$ is $G$-equivariant. Hence, $A$ has the $G$-WEP.
\end{proof}

A $C^*$-algebra is said to have the QWEP if it is a quotient of a $C^*$-algebra with the WEP. In \cite{K}, Kirchberg conjectured that every $C^*$-algebra has the QWEP. Until recently, this conjecture stood as one of the most important open problems in operator algebra theory, where it was shown to be false \cite{jietal}. In the setting of $G$-$C^*$-algebras, we can simply show that the $G$-equivariant analogue of the QWEP conjecture is false. Say a $G$-$C^*$-algebra has the \emph{$G$-QWEP} if it is the image under a $G$-equivariant $*$-homomorphism of a $G$-C$^*$-algebra with the $G$-WEP.

\begin{prop}
Let $G$ be a non-amenable locally compact group. Then no $C^*$-algebra with a $G$-invariant state has the $G$-QWEP. In particular, $C^*(G)$ does not have the $G$-QWEP.
\end{prop}

\proof
Let $A$ be a $C^*$-algebra admitting a $G$-invariant state $\varphi$, and suppose for contradiction that $q: B \to A$ is a $G$-equivariant surjective $*$-homomorphism, where $B$ is a $G$-$C^*$-algebra with the $G$-WEP. Then there is a $G$-equivariant copy of $\cI_{L^1(G)}(B)$ in $B^{**}$ such that the inclusion $\cI_{L^1(G)}(B) \hookrightarrow B^{**}$ restricts to the canonical inclusion of $B$. Note that $1_{B^{**}} \in \cI_{L^1(G)}(B)$. (Indeed, if $\Phi: B^{**} \to \cI_{L^1(G)}(B)$ is a complete contraction such that $\Phi_B = \text{id}_B$, then each $b$ in $B$ is in the multiplicative domain of $\Phi$, so $\Phi(1_{B^{**}})b = \Phi(b)=b$. Thus $\Phi(1_{B^{**}}) = 1_{B^{**}}$.) So $\mathbb C \hookrightarrow \cI_{L^1(G)}(B)$ via $z \mapsto z 1_{B^{**}}$, which implies the existence of an $L^1(G)$-embedding $\cI_{L^1(G)}(\mathbb C) \hookrightarrow \cI_{L^1(G)}(B) \subseteq B^{**}$. Composing this inclusion with $\varphi^{**} \circ q^{**}: B^{**} \to \mathbb C$ yields an $L^1(G)$-module map $\cI_{L^1(G)}(\mathbb C) \to \mathbb C$ restricting to the identity on $\mathbb C$. By $L^1(G)$-rigidity, $\cI_{L^1(G)}(\mathbb C) = \mathbb C$, so that $\mathbb C$ is $L^1(G)$-injective. However, this would give the existence of a $G$-invariant state $L^\infty(G) \to \mathbb C$, thus contradicting non-amenability of $G$.
\endproof

The amenable radical of a locally compact group $G$ is the unique amenable normal subgroup of $G$ containing all other amenable normal subgroups of $G$. By \cite{KenRaum}, $C^*_\lm(G)$ admits a tracial state if and only if the amenable radical of $G$ is open (which of course holds in particular for discrete groups). Since a tracial state on $C^*_\lm(G)$ is necessarily $G$-invariant, we get the following as a corollary.

\begin{cor}
If $G$ is a non-amenable locally compact group with open amenable radical, then $C^*_\lm(G)$ does not have the $G$-QWEP.
\end{cor}

\subsection{The discrete group case}\label{s:gWEP}

Throughout this section $\Gamma$ is a discrete group, and by a $\Gamma$-$C^*$-algebra we mean a unital $C^*$-algebra $A$ on which $\Gamma$ acts by $*$-automorphisms. By a $\Gamma$-map we always mean a $\Gamma$-equivariant unital completely positive (ucp) map.

If $ A$ is a $\Gamma$-$C^*$-algebra, then as noted above there is a canonical action of $\ell^1(\Gamma)$ on $ A$ such that $ A$ with this action belongs to $\ell^1(\Gamma)\textbf{mod}$. Clearly, in this setting, $\Gamma$-maps and unital $\ell^1(\Gamma)$-morphisms coincide. It follows verbatim from Proposition \ref{p:G-inj} that a $\Gamma$-$C^*$-algebra $A$ is injective in $\ell^1(\Gamma)\textbf{mod}$ if and only if it is injective in the category of $\Gamma$-$C^*$-algebras with $\Gamma$-maps, the only difference being we consider unital morphisms. Similarly, by Proposition \ref{p:G-WEP} the $\ell^1(\Gamma)$-WEP for $A$ coincides with the analogous notion in the category of $\Gamma$-$C^*$-algebras with $\Gamma$-maps. Thus, we say that a $\Gamma$-$C^*$-algebra $ A$ has the $\Gamma$-WEP if it has $\ell^1(\Gamma)$-WEP.

Let $ A$ be a $C^*$-algebra. The $C^*$-algebra $\ell^\infty(\Gamma,  A)$ of bounded $ A$-valued functions on $\Gamma$ turns into a $\Gamma$-$C^*$-algebra (or a $\Gamma$-von Neumann algebra if $ A$ is a von Neumann algebra) with the action
\[
(t\cdot f)(s) = f(t^{-1}s)\quad\quad \left(s, t\in \Gamma \text{ and } f\in \ell^\infty(\Gamma,  A)\right) .
\]
The map $\iota: \ell^\infty(\Gamma)\to \ell^\infty(\Gamma,  A)$ defined by
\begin{equation}\label{linf-embed}
\iota(f)(s) = f(s) \mathds{1}_ A
\end{equation}
is a $\Gamma$-equivariant $C^*$- (or von Neumann-)embedding. In this setup, if $\varphi:  A_1 \to  A_2$ is a ucp map between $C^*$-algebras, then $\widetilde\varphi:\ell^\infty(\Gamma,  A_1) \to \ell^\infty(\Gamma,  A_2)$, $f \mapsto \varphi\circ f$ is a $\Gamma$-map. If $ A$ is injective, then $\ell^\infty(\Gamma,A)$ is injective in $\ell^1(\Gamma){\rm \textbf{mod}}$. The latter claim follows from a proof similar to that of Hamana's in \cite[Lemma 2.2]{Ham85}, where this statement is proved for slightly different categories. We won't need the following, but we note that the above works in more general setup. If $X \subseteq \cB(H)$ is an operator space, then $\ell^\infty(\Gamma,X)$ is an $\ell^1(\Gamma)$-submodule of $\ell^\infty(\Gamma,\cB(H))$. So in particular, $\ell^\infty(\Gamma,X)$ is in $\ell^1(\Gamma)\textbf{mod}$. If $X$ is injective, then a proof similar to the one mentioned above in \cite[Lemma 2.2]{Ham85} shows that $\ell^\infty(\Gamma,X)$ is injective in $\ell^1(\Gamma)\textbf{mod}$.

If $ A$ is a $\Gamma$-$C^*$-algebra and $\pi: A\to \cB(H)$ is a faithful nondegenerate $*$-representation, then the map
\begin{equation} \label{G-emb of A in ell^infty(G,B(H))}
j_\pi:  A\to \ell^\infty(\Gamma, \cB(H)), \quad j_\pi(a)(t) = \pi(t^{-1}a),
\end{equation}
for $a\in A$ and $t\in \Gamma$, is a $\Gamma$-embedding. Note that $\overline{j_\pi( A)}^{\text{weak*}}\subseteq \ell^\infty\left(\Gamma, \overline{\pi( A)}^{\text{weak*}}\right)$, where the weak*-closures are taken in $\ell^\infty(\Gamma, \cB(H))$, and $\cB(H)$, respectively.

In the following, we use $\cI_\Gamma( A)$ to denote Hamana's $\Gamma$-injective envelope of a $\Gamma$-$C^*$-algebra $ A$ (see \cite{Ham85}). This is the injective envelope of $ A$ in the category of $\Gamma$-operator systems with $\Gamma$-equivariant u.c.p.\ maps.

Some of the proofs below also use the injective envelope $\cI( A)$ of a $C^*$-algebra $ A$, and in particular, the fact that $ A$ has the WEP if and only if there is a completely isometric embedding $\cI( A) \hookrightarrow  A^{**}$ fixing elements in $ A$ (this is attributed to Blackadar in \cite[Introduction]{Paulsen11}).

\begin{thm} \label{G-WEP chars}
If $ A$ is a $\Gamma$-$C^*$-algebra, the following are equivalent:
\begin{enumerate}
\item $ A$ has the $\Gamma$-WEP.

\item For every $\Gamma$-$C^*$-algebra $B$ and completely isometric $\Gamma$-map $\kappa:  A \hookrightarrow   B$, there is a $\Gamma$-map $\psi:   B \to  A^{**}$ such that $\psi \circ \kappa = \iota_ A$.

\item For every faithful nondegenerate representation $\pi:  A \to \cB(H)$, there is a $\Gamma$-map $\Phi: \ell^\infty(\Gamma, \cB(H)) \to \overline{j_\pi( A)}^{\text{weak*}}\subseteq\ell^\infty(\Gamma, \cB(H))$ such that $\Phi(a) = a$ for all $a \in j_\pi( A)$.

\item Letting $\pi_u:  A \to B(H_u)$ denote the universal representation, there is a $\Gamma$-map $\Phi: \ell^\infty(\Gamma, B(H_u)) \to j_{\pi_u}( A)'' \cong  A^{**}$ such that $\Phi(a) = a$ for all $a \in j_{\pi_u}( A)$.

\item For every faithful covariant representation $\pi: (\Gamma \curvearrowright  A) \to \cB(H)$, there is a $\Gamma$-map $\Phi: \cB(H) \to \pi( A)''$ such that $\Phi(a) = a$ for all $a \in \pi( A)$.

\item Letting $(1 \otimes \lambda,\pi): (\Gamma \curvearrowright  A) \to B(H_u \otimes \ell^2(\Gamma))$ denote the regular representation induced from the universal representation of $ A$, there is a $\Gamma$-map $\Phi: B(H_u \otimes \ell^2(\Gamma)) \to \pi( A)'' \cong  A^{**}$ such that $\Phi(a) = a$ for all $a \in \pi( A)$.

\item There is a $\Gamma$-embedding $\cI_\Gamma( A) \hookrightarrow  A^{**}$.
\end{enumerate}
\end{thm}

\proof

(1) $\Longleftrightarrow$ (2): This follows immediately from Proposition \ref{p:G-WEP}.

(2) $\implies$ (3): Assume (2), and let $\pi:  A \to \cB(H)$ be a faithful representation. Since $j_\pi:  A \to \ell^\infty(\Gamma, \cB(H))$ is a completely isometric $\Gamma$-map, there is by (2) a $\Gamma$-map $\psi: \ell^\infty(\Gamma, \cB(H)) \to  A^{**}$ such that $\psi \circ j_\pi = \iota_ A$. By the universal property of $ A^{**}$, there is a normal $*$-homomorphism $\varphi:  A^{**} \to \overline{j_\pi( A)}^{\text{weak*}}$ such that $\varphi(a) = j_\pi(a)$ for all $a \in  A$. By normality, $\varphi$ is a $\Gamma$-map. Thus $\varphi \circ \psi$ satisfies the conclusion of (3).

(3) $\implies$ (4): The only thing not completely trivial about this implication is the assertion that $j_{\pi_u}( A)'' \cong  A^{**}$. This follows from the universal property of $ A^{**}$ together with the fact that the canonical inclusion $\tilde \pi_u:  A^{**} \hookrightarrow B(H_u)$ induces a normal injective $*$-homomorphism $j_{\tilde \pi_u} :  A^{**} \hookrightarrow \ell^\infty(\Gamma,B(H_u))$ that restricts to $j_{\pi_u}$ on $ A$.

(4) $\implies$ (7): This direction follows since $\ell^\infty(\Gamma, B(H_u))$ is injective in the category of $\Gamma$-$C^*$-algebras with $\Gamma$-maps (see \cite[Lemma 2.2]{Ham85}). Thus there is a $\Gamma$-map $\varphi:\cI_\Gamma( A) \to \ell^\infty(\Gamma, B(H_u))$ such that $\varphi(a) = j_{\pi_u}(a)$ for all $a \in  A$. With $\Phi$ from (4), the composition $\Phi \circ \varphi: \cI_\Gamma( A) \to  A^{**}$ gives a $\Gamma$-embedding by $\Gamma$-essentiality (see \cite[Section 2]{Ham85}).

(7) $\implies$ (5): Assume that there is a $\Gamma$-embedding $\kappa:\cI_\Gamma( A) \hookrightarrow  A^{**}$, and let $\pi: (\Gamma \curvearrowright  A) \to \cB(H)$ be a covariant representation. By $\Gamma$-injectivity, the inclusion $ A \hookrightarrow \cI_\Gamma( A)$ extends to a $\Gamma$-map $\tilde \kappa: \cB(H) \to \cI_\Gamma( A)$. By the universal property of $ A^{**}$, there is a $\Gamma$-map $\psi:  A^{**} \to \pi( A)''$ such that $\psi(a) = \pi(a)$ for all $a \in  A$. The composition $\psi \circ \tilde \kappa: \cB(H) \to \pi( A)''$ is the desired $\Gamma$-map.

(5) $\implies$ (6): This follows similarly to the implication (3) $\implies$ (4).

(6) $\implies$ (4): Since $\ell^\infty(\Gamma,B(H_u)) \subseteq B(H_u \otimes \ell^2(\Gamma))$ canonically as a $\Gamma$-subspace, the restriction of the map guaranteed by (6) satisfies the claim in (4).

(7) $\implies$ (2): This follows from a routine use of $\Gamma$-injectivity.
\endproof

\begin{thm}
If $\Gamma\ltimes A$ has the WEP then $ A$ has the $\Gamma$-WEP.
\end{thm}

\begin{proof}
Denote by $E: \Gamma\ltimes A \to  A$ the canonical conditional expectation, which is a $\Gamma$-map, and let $E^{**}: (\Gamma\ltimes A)^{**} \to  A^{**}$ be its second adjoint. So, using
\cite[Theorem 3.4]{Ham85}, we have the following $\Gamma$-maps
\[
\cI_\Gamma( A) \to \cI(\Gamma\ltimes A) \to (\Gamma\ltimes A)^{**}\to  A^{**}
\]
whose composition $\cI_\Gamma( A) \to  A^{**}$ restricts to
identity on $ A$, hence is a $\Gamma$-embedding by $\Gamma$-essentiality.
Thus, by Theorem \ref{G-WEP chars}, $ A$ has the $\Gamma$-WEP.
\end{proof}

\subsection{The WEP for actions}

A discrete group $\Gamma$ is said to act amenably on a $\Gamma$-$C^*$-algebra $ A$ if the bidual action $\Gamma\acts A^{**}$ is amenable. It is known that a group can act amenably on a non-nuclear $C^*$-algebra (or a non-injective von Neumann algebra). Motivated by the above characterization, and in order to include actions on $C^*$-algebras without the WEP, we make the following definition.

\begin{defn}
We say an action $\Gamma\acts  A$ has the WEP if there is a $\Gamma$-map
$\Phi: \ell^\infty(\Gamma, A^{**}) \to A^{**}$
such that $(\Phi \circ j)(a) = a$ for all
$a \in A$.
\end{defn}

Obviously, from the definitions, if $ A$ has the $\Gamma$-WEP then $\Gamma \acts A$ has the WEP.
\begin{prop}
If $\Gamma\acts A$ has the WEP and $ A$ has the WEP,
then $ A$ has the $\Gamma$-WEP.
\end{prop}

\begin{proof}
Since $ A$ has the WEP, there is a $\Gamma$-embedding $\iota: \ell^\infty(\Gamma,\cI( A)) \hookrightarrow \ell^\infty(\Gamma,  A^{**})$ that fixes elements in $\ell^\infty(\Gamma,  A)$. With $j:  A \to \ell^\infty(\Gamma, \cI( A))$ and $j_{\pi_u}:  A \to \ell^\infty(\Gamma,  A^{**})$ the $\Gamma$-embeddings described as in Equation \ref{G-emb of A in ell^infty(G,B(H))}, we have $\iota \circ j(a)=j_{\pi_u}(a)$. Since $\Gamma \acts A$ has the WEP, we can compose to get a $\Gamma$-map $\ell^\infty(\Gamma,\cI(A)) \to A^{**}$ restricting to the identity on the canonical copies of $A$. Since $\ell^\infty(\Gamma,\cI(A))$ is $\Gamma$-injective, the result follows from Theorem \ref{A-WEP-eq-conds} (6).
\end{proof}

\begin{prop} \label{amenable action has WEP}
If $\Gamma\acts A$ is amenable, then $\Gamma\acts A$ has the WEP.
\end{prop}

\begin{proof}

Consider the automorphism $\kappa: \ell^\infty(\Gamma,  A^{**}) \to \ell^\infty(\Gamma,  A^{**})$ defined by
$\kappa(f)(t) = t(f(t))$, $t\in\Gamma$.
For $s\in\Gamma$ and $f\in \ell^\infty(\Gamma,  A^{**})$ we have
\[
\kappa(sf)(t) = t (sf(t)) = t(f(s^{-1}t)) = s(s^{-1}t(f(s^{-1}t))) = 
s(\kappa(f)(s^{-1}t)) = s(s\kappa(f)(t)),
\]
which shows $\kappa(sf) = \alpha_s(\kappa(f))$,
where $\alpha:\Gamma\acts \ell^\infty(\Gamma,  A^{**})$ is the diagonal action.

Moreover, for every $a\in  A$,
\[
\kappa(j_{\pi_u}(a))(t) = t (j_{\pi_u}(a)(t)) = t (\pi_u(t^{-1}a)) = {\pi_u}(a),
\]
which shows $\kappa(j_{\pi_u}(a)) = 1\otimes a$.

The result is now clear by the definition of the WEP for the action and the fact that an action $\Gamma\acts A$ is amenable if and only if there is an $\alpha$-equivariant conditional expectation
$E: \ell^\infty(\Gamma,  A^{**}) \cong \ell^\infty(\Gamma) \overline{\otimes}  A^{**} \to 1\otimes  A^{**} \cong  A^{**}$.
\end{proof}

Combining the previous two propositions, we get:

\begin{cor}\label{amen+WEP-->G-WEP}
If $\Gamma\acts A$ is amenable and $ A$ has the WEP, then $ A$ has the $\Gamma$-WEP.
\end{cor}

\begin{example}
Let $\partial_F\Gamma$ denote the Furstenberg boundary of $\Gamma$.
Then $C(\partial_F\Gamma)$ has the $\Gamma$-WEP.
This follows by Theorem \ref{G-WEP chars} and the fact that
$\cI_\Gamma(C(\partial_F\Gamma)) = C(\partial_F\Gamma)$.
In particular, if $\Gamma$ is not exact, then $C(\partial_F\Gamma)$ has the $\Gamma$-WEP but $\Gamma\acts C(\partial_F\Gamma)$ is not amenable. 
\end{example}

\begin{thm} \label{G amen <-> C*rG has G-WEP} \label{WEP-Lance-dis}
The following are equivalent for a discrete group $\Gamma$.
\begin{enumerate}
\item
$\Gamma$ is amenable.
\item
$C^*_\lm(\Gamma)$ has the WEP.
\item
$C^*_\lm(\Gamma)$ has the $\Gamma$-WEP.
\end{enumerate}
\end{thm}

\begin{proof}
The equivalence of (1) and (2) is Lance's theorem \cite{L}. Since $\ell^1(\G)^*=\ell^\infty(\G)$ is injective, the implication (3) $\implies$ (2) follows from Theorem \ref{AWEP implies WEP char}. Now suppose $\Gamma$ is amenable and $C^*_\lm(\Gamma)$ has the WEP. Then the action $\Gamma\acts C^*_\lm(\Gamma)$ is amenable. Hence $C^*_\lm(\Gamma)$ has the $\Gamma$-WEP by Proposition \ref{amen+WEP-->G-WEP}.
\end{proof}

We will prove in the next section that the above conditions are also equivalent to $C^*_\lm(\Gamma)$ having the $A(\Gamma)$-WEP.

In the case of a general locally compact group $G$, the implication (2) $\Rightarrow$ (1) above is known to not hold in general, e.g., for $G=SL(2,\bR)$.

\subsection{On the kernel of actions with the WEP}

Let $R(\Gamma)$ denote the amenable radical of $\Gamma$, that is, the unique amenable normal subgroup of $\Gamma$ that contains all other amenable normal subgroups of $\Gamma$. By \cite[Proposition 7]{Fur} and \cite[Theorem 3.11]{KalKen}, $R(\Gamma) = \ker\left(\Gamma\act \cI_\Gamma(\bC)\right)$. So $\ker\left(\Gamma\act \cI_\Gamma( A)\right)$ may be considered as a generalized amenable radical of $\Gamma$ corresponding to the action $\Gamma \curvearrowright  A$.

\begin{prop} \label{radical containment}
For any $\Gamma$-$C^*$-algebra $ A$, $\ker\left(\Gamma\act \cI_\Gamma( A)\right) \subseteq R(\Gamma)$. In particular, $\ker\left(\Gamma\act \cI_\Gamma( A)\right)$ is amenable.
\end{prop}

\proof
Since $\cI_\Gamma(\mathbb C) \subseteq \cI_\Gamma( A)$ via a $\G$-embedding, the result follows from the fact mentioned above that $R(\Gamma)$ is the kernel of the action $\Gamma \curvearrowright \cI_\Gamma(\mathbb C)$.
\endproof

Evidently, the inclusion $\ker\left(\Gamma\act \cI_\Gamma( A)\right) \subseteq \ker\left(\Gamma\act  A\right)$ always holds. The converse inclusion also clearly holds if $ A$ is $\Gamma$-injective. We show below that it holds under the generally weaker assumption that $ A$ has the $\Gamma$-WEP.

\begin{thm} \label{GWEP implies equal kernels}
If $ A$ has the $\Gamma$-WEP, then $\ker\left(\Gamma\act \cI_\Gamma( A)\right) = \ker\left(\Gamma\act  A\right)$.
\end{thm}

\proof
As noted above, $\ker\left(\Gamma\act \cI_\Gamma( A)\right) \subseteq \ker\left(\Gamma\act  A\right)$ always holds.

For the other inclusion, suppose $ A$ has the $\Gamma$-WEP. Then there is a $\Gamma$-embedding $\vphi: \cI_\Gamma( A) \hookrightarrow  A^{**}$ such that $\vphi$ restricts to the canonical inclusion on $ A$. If $s \in \ker\left(\Gamma\act  A\right)$, then by weak*-continuity, $s$ is also in the kernel of the action $\Gamma \curvearrowright  A^{**}$. Thus $s$ also acts trivially on everything in $\cI_\Gamma( A)$, i.e., $s \in \ker\left(\Gamma\act \cI_\Gamma( A)\right)$.
\endproof

The converse of Theorem \ref{GWEP implies equal kernels} is not true: if $\Gamma$ is trivial, then $\ker\left(\Gamma\act \cI_\Gamma( A)\right) =\ker\left(\Gamma\act  A\right)=\Gamma$ for all $ A$.

\begin{example}
\begin{enumerate}
\item
Consider the trivial action $\Gamma \curvearrowright \mathbb C$. Then $\ker\left(\Gamma\act \cI_\Gamma(\bC)\right) = R(\Gamma)$ and $\ker\left(\Gamma\act \mathbb{C}\right) = \Gamma$. So for $\Gamma$ non-amenable, this gives an example when $\ker\left(\Gamma\act \cI_\Gamma( A)\right) \not= \ker\left(\Gamma\act  A\right)$.
\item
Consider the canonical action $\Gamma \curvearrowright \ell^\infty(\Gamma)$. Since $\ell^\infty(\Gamma)$ is $\Gamma$-injective, it follows $\ker\left(\Gamma\act \cI_\Gamma(\ell^\infty(\Gamma))\right) = \ker\left(\Gamma\act \ell^\infty(\Gamma)\right)$, and it is easy to see that these are both the trivial subgroup. So for $\Gamma$ non-amenable, this gives an example when $\ker\left(\Gamma\act \cI_\Gamma( A)\right) \not= R(\Gamma)$.
\item
Consider the canonical inner action $\Gamma \curvearrowright C^*_\lambda(\Gamma)$. It is straightforward to check that $\ker\left(\Gamma\act C^*_\lambda(\Gamma)\right)$ is the center $\mathcal Z(\Gamma)$. If $\Gamma$ is amenable, then $C^*_\lambda(\Gamma)$ has the $\Gamma$-WEP by Theorem \ref{G amen <-> C*rG has G-WEP}, so we get $\ker\left(\Gamma\act \cI_\Gamma(C^*_\lambda(\Gamma))\right) = \mathcal Z(\Gamma)$ by Theorem \ref{GWEP implies equal kernels}. Hence for $\Gamma$ amenable and non-abelian, this gives an example when $\ker\left(\Gamma\act \cI_\Gamma( A)\right)$ is nontrivial and not equal to $R(\Gamma)$.
\end{enumerate}
\end{example}

\begin{cor}
If $\Gamma$ is a discrete group with trivial amenable radical and $ A$ is a $\Gamma$-$C^*$-algebra with the $\Gamma$-WEP, then the action $\Gamma \curvearrowright  A$ is faithful.
\end{cor}

\proof
If $R(\Gamma)$ is trivial, then so is $\ker\left(\Gamma\act \cI_\Gamma( A)\right)$ by Proposition \ref{radical containment}. The result thus follows from Theorem \ref{GWEP implies equal kernels}.
\endproof

\subsection{Amenability of stabilizers}The theorem below can be considered as a generalization of Lance's result \cite{L} that $C^*_\lambda(\Gamma)$ has the WEP iff $\Gamma$ is amenable.

\begin{thm}
Let $\Gamma\acts X$ be a minimal action on a compact Hausdorff space $X$ such that $\Gamma\acts C(X)$ has the WEP. Then the stabilizer of any point $x\in X$ is amenable.

In particular, if $\Gamma\ltimes C(X)$ has the WEP then the stabilizer of any point $x\in X$ is amenable.
\end{thm}

\begin{proof}
Let $x\in X$, and let $\Lambda=\{g\in \G: gx = x\}$ be the stabilizer subgroup of $x$. Denote by $\cP_x: C(X) \to \ell^\infty(\Gamma)$ the injective 
*-homomorphism $\cP_x(f)(g) := f(gx)$ for $g\in \Gamma$.
For any $g\in \Gamma$, $h\in\Lambda$, and $f\in C(X)$ we have 
\[
\cP_x(f)(gh) = f(ghx) = f(gx) = \cP_x(f)(g) ,
\]
which shows $\cP_x(C(X))\subset \ell^\infty(\Gamma/\Lambda)$. Hence this yields a faithful representation $\pi: C(X) \to B\left(\ell^2(\Gamma/\Lambda)\right)$, and we have $$\overline{j_\pi(C(X))}^{\text{weak*}}\subset j_\id\left(\ell^\infty(\Gamma/\Lambda)\right),$$ where $\id: \ell^\infty(\Gamma/\Lambda) \to B\left(\ell^2(\Gamma/\Lambda)\right)$ is the canonical inclusion. Now suppose the action $\Gamma\acts C(X)$ has the WEP, that is, there exists a $\Gamma$-map $\Phi: \ell^\infty(\Gamma, \pi(C(X))'') \to \overline{j_\pi(C(X))}^{\text{weak*}}$. Then $\varphi=\delta_\Lambda\circ j_\id^{-1}\circ\Phi\circ\iota$ is a $\Lambda$-invariant state on $\ell^\infty(\Gamma)$, where $\iota: \ell^\infty(\Gamma) \to \ell^\infty\left(\Gamma, B\left(\ell^2(\Gamma/\Lambda)\right)\right)$ is the $\Gamma$-embedding defined in \eqref{linf-embed}. Hence, $\Lambda$ is amenable.
\end{proof}

\section{Amenable actions and the $A(G)$-WEP}

In this section we characterize amenable dynamical systems through the $A(G)$-module structure of their associated crossed products. Our techniques are streamlined using the perspective of co-actions, so we begin with a quick overview of the relevant definitions. This perspective also hints at potential quantum group generalizations of our results. We postpone this investigation for future work.

Let $G$ be a locally compact group. The adjoint of convolution $\LO\pten\LO\rightarrow\LO$ is a co-associative co-multiplication $\Delta:\LI\rightarrow\LI\oten\LI$ satisfying $\Delta(f)(s,t)=f(st)$, for all $f\in\LI$. There are left and right fundamental unitaries $W,V\in\mc{B}(L^2(G\times G))$ which implement $\Delta$ in the sense that
$$\Delta(f)=W^*(1\ten M_f)W=V(M_f\ten 1)V^*, \ \ \ f\in\LI.$$
They are given respectively by
$$W\xi(s,t)=\xi(s,s^{-1}t), \ \ \  V\xi(s,t)=\xi(st,t)\delta_G(t)^{1/2}, \ \ \ s,t\in G, \ \xi\in L^2(G\times G),$$
where $\delta_G$ is the modular function. These fundamental unitaries are intimately related to the left and right regular representations $\lm,\rho:G\rightarrow\BLT$ given by
$$\lm(s)\xi(t)=\xi(s^{-1}t), \ \ \rho(s)\xi(t)=\xi(ts)\delta_G(s)^{1/2}, \ \ \ s,t\in G, \ \xi\in\LT.$$
It follows that $W\in\LI\oten VN(G)$ and $V\in VN(G)'\oten\LI$.

The adjoint of pointwise multiplication $A(G)\pten A(G)\rightarrow A(G)$ defines a co-associative co-multiplication $\wh{\Delta}:VN(G)\rightarrow VN(G)\oten VN(G)$ satisfying $\wh{\Delta}(\lm(s))=\lm(s)\ten\lm(s)$, $s\in G$. There are left and right fundamental unitaries $\h{W},\h{V}\in\mc{B}(L^2(G\times G))$ which implement the co-product via
$$\h{\Delta}(x)=\h{W}^*(1\ten x)\h{W}=\h{V}(x\ten 1)\h{V}^*, \ \ \ x\in VN(G).$$
They are given specifically by 
$$\h{W}\xi(s,t)=\xi(ts,t), \ \ \ \h{V}\xi(s,t)=W\xi(s,t)=\xi(s,s^{-1}t), \ \ \ s,t\in G, \ \xi\in L^2(G\times G).$$
Note that $\h{W} \in VN(G) \overline{\otimes} L^\infty(G)$. 

Both co-products $\Delta$ and $\h{\Delta}$ admit left and right extensions to $\BLT$ via the fundamental unitaries. For instance,
$$\Delta^l: \BLT\ni T\mapsto W^*(1\ten T)W\in \LI\oten\BLT$$
and
\begin{equation}\label{e:rightext}\Delta^r:\BLT\ni T\mapsto V(T\ten 1)V^*\in\BLT\oten\LI.\end{equation}
These maps, in turn, yield operator $\LO$-module structures on $\BLT$ given by
$$T\cdot f = (f\ten\id)\Delta^l(T) = \int_G f(s)\lm(s)^*T\lm(s) \ ds$$
and
$$f\cdot T = (\id\ten f)\Delta^r(T) = \int_Gf(s)\rho(s)T\rho(s)^* \ ds,$$
for $f\in\LO$ and $T\in\BLT$. Analogously, the lifted co-products $\h{\Delta}^l$ and $\h{\Delta}^r$ yield an operator $A(G)$-bimodule structure on $\BLT$. In this case the left and right module structures coincide since $\h{V}=W=\sigma\h{W}^*\sigma$. Indeed, for every $u\in A(G)$ and $T\in\BLT$ we have
\begin{align*}T\cdot u&=(u\ten\id)\h{\Delta}^l(T)=(u\ten\id)\h{W}^*(1\ten T)\h{W}\\
&=(\id\ten u)W(T\ten 1)W^*\\
&=(\id\ten u)\h{V}(T\ten 1)\h{V}\\
&=u\cdot T
\end{align*}
When $G$ is discrete, the $A(G)$-module action is nothing but Schur product with the matrix $[u(st^{-1})]_{s,t\in G}$, where $u\in A(G)$.

Viewing the extended co-multiplication $\Delta^r$ as a map $\BLT\rightarrow\BLT\oten\BLT$, its pre-adjoint defines a completely contractive multiplication on the space of trace-class operators $\TC$ via
\begin{equation*}\rhd:\TC\pten\TC\ni\om\ten\tau\mapsto\om\rhd\tau=\Delta^r_*(\om\ten\tau)\in\TC.\end{equation*}
The resulting bimodule structure on $\BLT$ satisfies
$$T\rhd\rho=(\rho\ten\id)V(T\ten 1)V^*, \ \ \ \rho\rhd T=(\id\ten\rho)V(T\ten 1)V^*=\pi(\rho)\cdot T$$
for $T\in\BLT$, $\rho\in\TC$, where $\pi:\TC\twoheadrightarrow\LO$ is the canonical quotient map given by restriction to $\LI$. Hence, the left $\rhd$-module structure degenerates to the left $\LO$-module structure defined above. The analogous construction exists for $\h{\Delta}^r$, yielding a dual product $\h{\rhd}$ on $\TC$ and a corresponding $\h{\rhd}$-bimodule structure on $\BLT$. Dually to $\rhd$, the left $\h{\rhd}$-module structure degenerates to the left $A(G)$-action.

If $(M,G,\alpha)$ is a $W^*$-dynamical system, the induced normal $*$-homomorphism $\alpha:M\rightarrow\LI\oten M$, $\alpha(x)(s)=\alpha_{s^{-1}}(x)$, $x\in M$, $s\in G$,
is co-associative in the sense that
$$(\Delta\ten\id)\circ \alpha = (\id\ten\alpha)\circ \alpha.$$
The corresponding $\LO$-module structure is determined by
$$x\star f=(f\ten\id)\alpha(x)=\int_G f(s)\alpha_{s^{-1}}(x), \ \ \ f\in\LO, \ x\in M.$$

It follows from the fixed point description of $G\bar{\ltimes}M$ (see, e.g., \cite[Theorem 16.1.15]{Ren}) that 
$$G\bar{\ltimes}M=\{T\in\BLT\oten M\mid (\id\ten\alpha_{s^{-1}})(T)=(\Ad(\rho(s))\ten\id)(T) \ \forall \ s\in G\}.$$
By point-weak* continuity of $W^*$-dynamical systems, the condition $(\id\ten\alpha_{s^{-1}})(T)=(\Ad(\rho(s))\ten\id)(T)$ for all $s\in G$ is equivalent to 
$$(\id\ten f\ten\id)(\id\ten\alpha)(T)=(\id\ten f\ten\id)(\Delta^r\ten\id)(T), \ \ \ f\in\LO$$
(by approximating point masses by suitable nets in $\LO$). Hence,
$$G\bar{\ltimes} M=\{X\in\BLT\oten M\mid (\id\ten\alpha)(X)=(\Delta^r\ten\id)(X)\}$$

The system $(M,G,\alpha)$ admits a dual co-action 
$$\h{\alpha}:G\bar{\ltimes}M\rightarrow VN(G)\oten (G\bar{\ltimes}M)$$
of $VN(G)$ on the crossed product, given by
\begin{equation}\label{e:coaction}\h{\alpha}(X)=(\h{W}^*\ten 1)(1\ten X)(\h{W}\ten 1), \ \ \ X\in G\bar{\ltimes}M.\end{equation}
On the generators we have $\h{\alpha}(\hat{x}\ten 1)=(\h{W}^*(1\ten\hat{x})\h{W})\ten 1$, $\hat{x}\in VN(G)$ and $\h{\alpha}(\alpha(x))=1\ten\alpha(x)$, $x\in M$. Moreover, by \cite[Theorem 2.7]{Vaes}
$$(G\bar{\ltimes}M)^{\h{\alpha}}=\{X\in G\bar{\ltimes}M \mid \h{\alpha}(X)=1\ten X\}=\alpha(M).$$
This co-action yields a canonical right operator $A(G)$-module structure on the crossed product $G\bar{\ltimes}M$ via
$$X\cdot u=(u\ten\id)\h{\alpha}(X), \ \ \ X\in G\bar{\ltimes}M, \ u\in A(G).$$

Assuming $M$ is standardly represented on $H$, there exists a strongly continuous unitary representation $u:G\rightarrow \BH$ and corresponding generator $U\in\LI\oten\BH$ such that $\alpha(x)=U^*(1\ten x)U$, $x\in M$ \cite[Corollary 3.11]{H}. At the level of vectors $\xi\in C_c(G,H)\subseteq \LT\ten H$ we have
\begin{equation}\label{e:1}\alpha(x)\xi(s)=U^*(1\ten x)U\xi(s)=u_s^*xu_s\xi(s), \ \ \ s\in G, \ x\in M.\end{equation}

A $W^*$-dynamical system $(M,G,\alpha)$ is \textit{amenable} if there exists a projection of norm one $P:\LI\oten M\rightarrow M$ such that $P\circ(\lambda_s\ten \alpha_s)=\alpha_s\circ P$, $s\in G$, where $\lambda$ also denotes the left translation action on $\LI$. For example, $(\LI,G,\lm)$ is always amenable, and $G$ is amenable if and only if the trivial action $G\acts\{x_0\}$ is amenable, in which case $P$ becomes a left invariant mean on $\LI$. When $G$ is second countable and $M=L^\infty(X,\mu)$ for a regular $G$-space $(X,\mu)$ (see \cite[Definition 2.1.1]{Mon} for a definition), it is known that that $(M,G,\alpha)$ is amenable if and only if $M$ is relatively $G$-injective \cite[Theorem 5.7.1]{Mon}. At the level on $\LO$-modules, we now show that this equivalence holds in general. 

\begin{prop}\label{p:Zimmer} A $W^*$-dynamical system $(M,G,\alpha)$ is amenable if and only if $M$ is relatively $\LO$-injective. 
\end{prop}

\begin{proof} Assume $M$ is standardly represented in $\BH$, and let $U\in \LI\oten\BH$ be the unitary implementation of $\alpha$. By commutativity of $\LI$ and formula (\ref{e:1}) it follows that both $\Ad(U)$ and $\Ad(U^*)$ leave $\LI\oten M$ invariant. Moreover, for every $\xi,\eta\in C_c(G,H)$, $h\in\LI$, $x\in M$, we have
\begin{align*}&\la U^*(\lambda_s\ten\alpha_s(h\ten x))U\xi,\eta\ra=\int_G\la U^*(\lambda_s\ten\alpha_s(h\ten x))U\xi(t),\eta(t)\ra \ dt\\
&=\int_G h(s^{-t}t)\la u_{t^{-1}}\alpha_s(x)u_t\xi(t),\eta(t)\ra \ dt=\int_G h(s^{-t}t)\la u_{t^{-1}s}xu_{s^{-1}t}\xi(t),\eta(t)\ra \ dt\\
&=\int_G h(t)\la u_{t^{-1}}xu_{t}\xi(st),\eta(st)\ra \ dt=\la U^*(h\ten x)U(\lm_{s^{-1}}\ten 1)\xi,(\lm_{s^{-1}}\ten 1)\eta\ra\\
&=\la (\lambda_s\ten\id)(U^*(h\ten x)U)\xi,\eta\ra.
\end{align*}
It follows that $\Ad(U^*)\circ(\lambda_s\ten\alpha_s)=(\lambda_s\ten\id)\circ\Ad(U^*)$. 

Now, suppose $(M,G,\alpha)$ is amenable. Then there exists a projection $P:\LI\oten M\rightarrow M$ of norm one such that $P\circ(\lambda_s\ten \alpha_s)=\alpha_s\circ P$, $s\in G$. Then $Q:=P\circ \Ad(U)$ satisfies
$$Q\circ(\lambda_s\ten\id)=P\circ\Ad(U)\circ(\lambda_s\ten\id)=P\circ(\lambda_s\ten\alpha_s)\circ \Ad(U)=\alpha_s\circ Q, \ \ \ s\in G,$$
and $Q(\alpha(x))=P(1\ten x)=x$. That is, $Q:\LI\oten M\rightarrow M$ is a completely contractive $G$-equivariant left inverse to $\alpha$. The composition $\alpha\circ Q$ is then a $G$-equivariant projection of norm one from $\LI\oten M$ onto the $G$-invariant von Neumann subalgebra $\alpha(M)$. By \cite[Lemme 2.1]{ADII}, there exists an $\LO$-equivariant projection of norm one $\Psi:\LI\oten M\rightarrow\alpha(M)$. The map $\alpha^{-1}\circ\Psi:\LI\oten M\rightarrow M$ is then an $\LO$-module left inverse to $\alpha$. Noting that under the canonical identification $\mathcal{CB}(\LO,M)=(\LO\pten M_*)^*=\LI\oten M$, the action $\alpha$ is the canonical embedding $j_M:M\rightarrow\mathcal{CB}(\LO,M)$. Since $M$ is a faithful $\LO$-module, it follows that $M$ is relatively injective over $\LO$.

Conversely, relative $\LO$-injectivity of $M$ entails the existence of a completely contractive $\LO$-morphism $\Phi:\LI\oten M\rightarrow M$ such that $\Phi\circ\alpha=\id_M$. It follows in a similar manner to the previous paragraph that there is a $G$-equivariant completely contractive left inverse $\Psi$ to $\alpha$. Define $P:\LI\oten M\ni X\mapsto \Psi(U^*XU)\in M$. Then 
$$P\circ (\lambda_s\ten\alpha_s)=\Psi\circ\Ad(U^*)\circ(\lambda_s\ten\alpha_s)=\Psi(\lambda_s\ten\id)\circ\Ad(U^*)=\alpha_s\circ P,$$
and $P(1\ten x)=\Psi(\alpha(x))=x$, so that $P$ is a projection of norm witnessing the amenability of $(M,G,\alpha)$.
\end{proof}

We now establish a perfect duality between amenability of $W^*$-dynamical systems $(M,G,\alpha)$ with $M$ injective and $A(G)$-injectivity of the associated crossed product $G\bar{\ltimes}M$. This generalizes \cite[Corollary 5.3]{C}, which corresponds to $M=\bC$, in which case $G\bar{\ltimes}M=VN(G)$. By virtually the same argument, the equivalence of (2) and (3) below persists to the level of co-actions of co-amenable locally compact quantum groups on von Neumann algebras. This will appear in future work. A similar result for actions of discrete quantum groups on von Neumann algebras was obtained independently in \cite{M}.

\begin{thm}\label{t:A(G)inj}
Let $(M,G,\alpha)$ be a $W^*$-dynamical system with $M$ injective. The following conditions are equivalent:
\begin{enumerate}
\item $(M,G,\alpha)$ is amenable;
\item $G\bar{\ltimes} M$ is $A(G)$-injective;
\item $M$ is $\LO$-injective.
\end{enumerate}
\end{thm}

\begin{proof} $(1)\Longrightarrow(2)$ By Proposition \ref{p:Zimmer}, there exists a completely contractive right $\LO$-module map $P:\LI\oten M\rightarrow M$ such that $P\circ \alpha=\id_M$. Define $\Theta(P):\BLT\oten M\rightarrow\BLT\oten M$ by
$$\Theta(P)=(\id\ten P)(\Delta^r\ten\id),$$
where $\Delta^r$ is the right extension of the co-product on $\LI$ (see (\ref{e:rightext})). It follows that
$$\la\Theta(P)(T),\rho\ten\om\ra=\la P(T\rhd\rho),\om\ra, \ \ \ T\in\BLT\oten M, \ \rho\in\TC, \ \om\in M_*,$$
where we abuse notation by letting $\rhd$ also denote the right action of $\TC$ on the first leg of $\BLT\oten M$, that is, $T\rhd\rho=(\rho\ten\id\ten\id)(\Delta^r\ten\id)(T)$. With this representation it is clear that $\Theta(P)$ is a right $\rhd$-module map. The argument from \cite[Proposition 4.2]{C} (applied to commutative quantum groups) amplifies to $\BLT\oten M$ and shows that $\Theta(P)$ is a left $\h{\rhd}$-module map, where $\rho\h{\rhd}T=(\id\ten\rho\ten\id)(\h{\Delta}^r\ten\id)(T)$. Hence, $\Theta(P)$ is a left (equivalently, right) $A(G)$-module map where the $A(G)$-action is on the first leg of $\BLT\oten M$.

The module property of $P$ is equivalent to $\alpha\circ P=(\id\ten P)(\Delta\ten\id)$, so for every $X\in\BLT\oten M$,
\begin{align*}(\id\ten\alpha)(\Theta(P)(X))&=(\id\ten\alpha)(\id\ten P)(\Delta^r\ten\id)(X)\\
&=(\id\ten\id\ten P)(\id\ten\Delta\ten\id)(\Delta^r\ten\id)(X)\\
&=(\id\ten\id\ten P)(\id\ten\Delta^r\ten\id)(\Delta^r\ten\id)(X)\\
&=(\id\ten\id\ten P)(\Delta^r\ten\id\ten\id)(\Delta^r\ten\id)(X)\\
&=(\Delta^r\ten\id)(\id\ten P)(\Delta^r\ten\id)(X)\\
&=(\Delta^r\ten\id)(\Theta(P)(X)).\end{align*}
On the one hand, this implies
\begin{align*}\Theta(P)(X)&=(\id\ten P)(\id\ten\alpha)(\Theta(P)(X))=(\id\ten P)(\Delta^r\ten\id)(\Theta(P)(X))\\
&=\Theta(P)(\Theta(P)(X),\end{align*}
so that $\Theta(P)$ is a projection of norm one. On the other hand, since 
$$G\bar{\ltimes} M=\{X\in\BLT\oten M\mid (\id\ten\alpha)(X)=(\Delta^r\ten\id)(X)\},$$
the chain of equalities above entails $\Theta(P)(\BLT\oten M)\subseteq G\bar{\ltimes} M$. Moreover, if $X\in G\bar{\ltimes} M$ then
$$\Theta(P)(X)=(\id\ten P)(\Delta^r\ten\id)(X)=(\id\ten P)(\id\ten\alpha)(X)=X,$$
and we see that $\Theta(P)$ is an $A(G)$-equivariant projection of norm one onto $G\bar{\ltimes} M$.

Since $VN(G)$ admits an $A(G)$-invariant state \cite[Theorem 4]{Re}, it follows from (the proof of) \cite[Theorem 5.5]{CN} that $\BLT$ is injective as a left, and hence right, operator $A(G)$-module. Since $M$ is an injective von Neumann algebra, it follows that $\BLT\oten M$ is injective in $\mathbf{mod}A(G)$. Hence, $G\bar{\ltimes} M$ is injective in $\mathbf{mod}A(G)$ via $\Theta(P)$.

$(2)\Longrightarrow(3)$ If $G\bar{\ltimes}M$ is injective in $\mathbf{mod}A(G)$, then there exists an $A(G)$-equivariant projection of norm one $E:\BLT\oten M\rightarrow G\bar{\ltimes}M$. By the $A(G)$-module property we have (see \cite[Corollary 4.3]{C})
$$E(\LI\oten M)\subseteq \LI\oten M\cap G\bar{\ltimes} M=\alpha(M).$$
Hence $P=\alpha^{-1}\circ E:\LI\oten M\rightarrow M$ is a completely contractive left inverse to $\alpha$. Moreover, since $E$ is an $G\bar{\ltimes} M$-bimodule map, it is a $VN(G)\oten 1$-bimodule map, and it follows from \cite[Theorem 4.9]{CN} that $E$ is a right $\LO$-module map. Hence, $P$ is a right $\LO$-module left inverse to $\alpha$, implying $M$ is relatively injective in $\mathbf{mod}\LO$. Since $M$ is also an injective von Neumann algebra, (3) follows from \cite[Proposition 2.3]{C}. 

$(3)\Longrightarrow(1)$ If $M$ is $\LO$-injective then it is relatively $\LO$-injective. Whence, $(M,G,\alpha)$ is amenable by Proposition \ref{p:Zimmer}.
\end{proof} 

As an application of the above results, we obtain a generalized version of \cite[Proposition 4.1]{ADI} to the locally compact setting (when the cocycle is trivial). Note that it is precisely the $A(G)$-module structure in condition $(2)$ which allows for the generalized version: as remarked in \cite[Remarque 4.4 (b)]{ADI}, the verbatim generalization of \cite[Proposition 4.1]{ADI} to the locally compact setting does not hold.

\begin{cor}\label{c:5.3} Let $(M,G,\alpha)$ be a $W^*$-dynamical system. The following conditions are equivalent:
\begin{enumerate}
\item $(M,G,\alpha)$ is amenable;
\item There exists an $A(G)$-module projection of norm one from $\BLT\oten M$ onto $G\bar{\ltimes} M$;
\item For every extension $(N,G,\beta)$ of $(M,G,\alpha)$, there exists a $G$-equivariant projection of norm one from $N$ onto $M$.
\end{enumerate}
\end{cor}

\begin{proof} The equivalence of $(1)$ and $(2)$ follows from the proof of Theorem \ref{t:A(G)inj}. The only difference, here, is that $M$ is not necessarily injective. In this case, condition (2) implies the relative $\LO$-injectivity of $M$, which implies (1) by Proposition \ref{p:Zimmer}.

$(1)\Longrightarrow(3)$ If $(N,G,\beta)$ is any extension of $(M,G,\alpha)$, then by definition (see \cite[D\'{e}finition 3.3]{ADI} there is a linear projection of norm one from $N$ onto $M$. By Proposition \ref{p:Zimmer}, $M$ is relatively $\LO$-injective, so there is an $\LO$-module projection $P$ of norm one from $N$ onto $M$. Since the continuous parts $N^c$ and $M^c$ of $N$ and $M$ coincide with $\la N\ast \LO\ra$ and $\la M\ast\LO\ra$, respectively (see \cite[Lemma 7.5.1]{Ped}), it follows that $P$ induces a $G$-equivariant projection of norm one from $N^c$ onto $M^c$. Hence, (3) follows from \cite[Lemme 2.1]{ADII}.

$(3)\Longrightarrow(1)$ Simply apply $(3)$ to the extension $(\LI\oten M,G,\lm\ten\alpha)$ of $(M,G,\alpha)$.

\end{proof}

Buss, Echterhoff and Willett recently introduced the following notion of amenability for $C^*$-dynamical systems \cite{BEW3}: $(A,G,\alpha)$ is \textit{amenable} if there exists a net of norm-continuous, compactly supported, positive type functions $h_i:G\rightarrow Z(A_\alpha'')$ such that $\norm{h_i(e)}\leq 1$ for all $i$, and $h_i(s)\rightarrow 1$ weak* in $A_\alpha''$, uniformly for $s$ in compact subsets of $G$. It was subsequently shown by the authors of this paper that this notion coincides with amenability of the universal enveloping system $(A_\alpha'',G,\overline{\alpha})$ \cite[Theorem 4.2]{BC}, and, for commutative systems $(C_0(X),G,\alpha)$, coincides with topological amenability of the transformation group $(G,X)$ \cite[Corollary 4.12]{BC}. 

In \cite[Proposition 7.10]{BEW3}, it was shown that for $C^*$-dynamical systems $(A,G,\alpha)$ with $A$ nuclear, amenability implies that $A$ has the $G$-WEP (equivalently, the $\LO$-WEP by Propositions \ref{p:G-WEP} and \ref{p:4.5}), and that both conditions are equivalent if $G$ is exact. We now complement this result by establishing a $C^*$-analogue of Theorem \ref{t:A(G)inj}.

\begin{thm}\label{t:A(G)WEP} Let $(A,G,\alpha)$ be a $C^*$-dynamical system with $A$ nuclear. Consider the following conditions:
\begin{enumerate}
\item $(A,G,\alpha)$ is amenable;
\item $G\ltimes A$ has the $A(G)$-WEP;
\item $A$ has the $\LO$-WEP.
\end{enumerate}
Then $(1)\Rightarrow(2)\Rightarrow(3)$, and if $G$ is exact, the conditions are equivalent.
\end{thm}

\begin{proof} $(1)\Longrightarrow(2)$: If $(A,G,\alpha)$ is amenable, then by \cite[Theorem 4.2]{BC} there exists a net $(h_i)$ of continuous compactly supported functions $h_i:G\rightarrow\mc{CB}(A)$, whose corresponding Herz-Schur multipliers $\Theta(h_i):G\ltimes A\rightarrow G\ltimes A$ (see \cite{MTT} and \cite{BC}) satisfy
$$\Theta(h_i)(\alpha\times\lm(f))=\int_G \alpha(h_i(s)(f(s)))(\lm_s\ten1) \ ds, \ \ \ f\in C_c(G,A),$$
$\norm{\Theta(h_i)}_{cb}\leq1$ and $\Theta(h_i)\rightarrow\id_{G\ltimes A}$ in the point norm topology. Note that each $\Theta(h_i)$ is an $A(G)$-morphism. 

As $G\ltimes_f A\cong G\ltimes A$ via the regular representation \cite[Proposition 5.9]{BEW3}, we have 
$$B(G\ltimes_f A)= (G\ltimes_f A)^*=(G\ltimes A)^*.$$
Since each $h_i$ is compactly supported, and compactly supported elements of $B(G\ltimes A)^+$ lie in $A(G\ltimes A)^+$ \cite[Lemma 7.7.6]{Ped}, it follows that $\Theta(h_i)^*$ maps $(G\ltimes A)^*$ into $A(G\ltimes A)=(G\ltimes A)''_*$. Representing $A$ faithfully inside $A_\alpha''$, it follows that $(G\ltimes A)''=G\bar{\ltimes}A_\alpha''$. Hence, the adjoint of the co-restriction of $\Theta(h_i)^*$ defines a completely contractive $A(G)$-morphism
$$\Phi_i:G\bar{\ltimes}A_\alpha''\rightarrow(G\ltimes A)^{**}.$$
Clustering the resulting net $(\Phi_i)$ to an $A(G)$-morphism $\Phi\in\mc{CB}(G\bar{\ltimes}A_\alpha'',(G\ltimes A)^{**})$, and appealing to the point norm convergence $\Theta(h_i)\rightarrow\id_{G\ltimes A}$, it follows that the diagram
\begin{equation*}
\begin{tikzcd}
&G\bar{\ltimes}A_\alpha''\arrow[rd, "\Phi"]\\
G\ltimes A\arrow[ru, hook ]\arrow[rr, hook] &  &(G\ltimes A)^{**}
\end{tikzcd}
\end{equation*}
commutes. Since $A$ is nuclear, $A_\alpha''$ is injective, and since $(A_\alpha'',G,\overline{\alpha})$ is amenable, Theorem \ref{t:A(G)inj} guarantees the $A(G)$-injectivity of $G\bar{\ltimes}A_\alpha''$. Hence, by Theorem \ref{A-WEP-eq-conds}, $G\ltimes A$ has the $A(G)$-WEP.

$(2)\Longrightarrow(3)$: If $G\ltimes A$ has the $A(G)$-WEP then the canonical inclusion $i:G\ltimes A\hookrightarrow\BLT\oten A_\alpha''$ is $A(G)$-flat, implying the existence of a $A(G)$-equivariant weak expectation $\BLT\oten A_\alpha''\rightarrow(G\ltimes A)^{**}$, which, by Lemma \ref{l:multiplier} is a completely positive $M(G\ltimes A)$-bimodule map. Composing with the canonical $A(G)$-morphism $(G\ltimes A)^{**}\twoheadrightarrow G\bar{\ltimes} A_\alpha''$ we obtain a completely positive $A(G)$-morphism $E:\BLT\oten A_\alpha''\rightarrow G\bar{\ltimes} A_\alpha''$ (which is also an $M(G\ltimes A)$-bimodule map) making the following diagram commute:
\begin{equation*}
\begin{tikzcd}
\BLT\oten A_\alpha''\arrow[rd, "E"]\\
G\ltimes A\arrow[u, hook ]\arrow[r, hook] & G\bar{\ltimes} A_\alpha''
\end{tikzcd}
\end{equation*}
The left fundamental unitary of $\LI$ satisfies $W\ten 1\in M(C_0(G)\iten (G\ltimes A))$
by the ``left handed version'' of \cite[Lemma 3.3]{LPRS}. Thus, for any $f\in\LO$, 
$$(f\ten\id\ten\id)(\id\ten E)(W\ten 1)=E((f\ten\id\ten\id)(W\ten 1))=(f\ten\id\ten\id)(W\ten 1),$$
as $(f\ten\id\ten\id)(W\ten 1)\in M(G\ltimes A)$. Then $(\id\ten E)(W\ten 1)=W\ten 1$, so that $W\ten 1$ lies in the multiplicative domain of the unital completely positive map $(\id\ten E)$. Then, as in \cite[Theorem 3.2]{NV}, the module property of unital completely positive maps over their multiplicative domains implies that
\begin{align*}E(T\cdot f)&=E((f\ten\id\ten\id)(W\ten 1)^*(1\ten T)(W\ten 1))\\
&=(f\ten\id\ten\id)(\id\ten E)((W\ten 1)^*(1\ten T)(W\ten 1)))\\
&=(f\ten\id\ten\id)((W\ten 1)^*(1\ten E(T))(W\ten 1))\\
&=E(T)\cdot f,
\end{align*}
where $T\cdot f$ is the action of $\LO$ on the left leg of $\BLT\oten A_\alpha''$.

Now, the dual co-action $\h{\overline{\alpha}}:G\bar{\ltimes} A_\alpha''\rightarrow VN(G)\oten (G\bar{\ltimes} A_\alpha'')$ on the crossed product satisfies $(G\bar{\ltimes}A_\alpha'')^{\h{\overline{\alpha}}}=\overline{\alpha}(A_\alpha'')$. Since $A(G)$ acts trivially on $\LI$, the $A(G)$-module property of $E$ implies
$$E|_{\LI\oten A_\alpha''}:\LI\oten A_\alpha''\rightarrow(G\bar{\ltimes} A_\alpha'')^{\h{\overline{\alpha}}}=\overline{\alpha}(A_\alpha''),$$
so we obtain a right $\LO$-module map $\Psi:\LI\oten A_\alpha''\rightarrow A^{**}$ via $\Phi\circ\overline{\alpha}^{-1}\circ E|_{\LI\oten A_\alpha''}$, where $\Phi:A_\alpha''\rightarrow A^{**}$ is the $\LO$-morphism constructed in the proof of Proposition \ref{p:4.5}. Since $\overline{\alpha}(A)=\alpha(A)\subseteq M(G\ltimes A)$, $E$ is an $M(G\ltimes A)$-bimodule map, and $\Phi|_{A}=i_{A}$ (see the proof of Proposition \ref{p:4.5}) the following diagram commutes

\begin{equation*}
\begin{tikzcd}
&\LI\oten A_\alpha''\arrow[rd, "\Psi"]\\
A\arrow[ru, "\overline{\alpha}|_{A}" ]\arrow[rr, hook, "i_{A}"] &  &A^{**}.
\end{tikzcd}
\end{equation*}

Since $A_\alpha''$ is an injective von Neumann algebra and $\LI$ is $\LO$-injective, it follows that $\LI\oten A_\alpha''$ is $\LO$-injective. Hence, the canonical inclusion $i_{A}:A\hookrightarrow A^{**}$ factors through an injective $\LO$-module, implying that $A$ has the $\LO$-WEP.

Finally, if $G$ is exact, the equivalence between the $G$-WEP and the $\LO$-WEP (Proposition \ref{p:G-WEP}) together with \cite[Proposition 7.10]{BEW3} shows that $(3)\Rightarrow(1)$.
\end{proof}

As a special case, we now establish the promised $A(G)$-equivariant analogue of Lance's theorem for discrete groups \cite{L}.

\begin{cor}\label{t:A(G)WEP} Let $G$ be a locally compact group. Then $G$ is amenable if and only if $C^*_\lm(G)$ has the $A(G)$-WEP.\end{cor}

\begin{proof} 
The forward direction follows immediately from $(1)\Rightarrow(2)$ in Theorem \ref{t:A(G)WEP} applied to $(\bC,G,\mathrm{trivial})$. Conversely, if $C^*_\lm(G)$ has the $A(G)$-WEP, then there is an $A(G)$-weak expectation $\Psi: \BLT \to C^*_\lm(G)^{**}$ that restricts to the identity map on $C^*_\lm(G)$. By Lemma \ref{l:multiplier} $\Psi$ is $M(C^*_\lm(G))-$, hence $G$-equivariant. Composing with the canonical map $C^*_\lm(G)^{**} \to VN(G)$, we get a $G$-equivariant $A(G)$-morphism $\Phi: B(L^2(G)) \to VN(G)$. Since $A(G)$ acts trivially on $L^\infty(G)$, we have $\Phi(L^\infty(G))\subset VN(G)^{\h{\Delta}} = \bC1$. Hence $\Phi|_{L^\infty(G)}$ defines an invariant mean, and $G$ is amenable.
\end{proof}

\section*{Acknowledgements}

The authors would like to thank Mehrdad Kalantar for his involvement in the early stages of the project. The second author was partially supported by the NSERC Discovery Grant RGPIN-2017-06275.

\end{spacing}

\end{document}